\documentclass[reqno, 12pt]{article}

\usepackage{enumerate}
\usepackage{needspace}
\usepackage{latexsym}
\usepackage{amsmath}
\usepackage{amsfonts}
\usepackage{amsthm}
\usepackage{amssymb,mathtools}
\usepackage{newlfont}
\usepackage{graphics}
\usepackage{color}
\usepackage{float}
\usepackage{diagbox}
\usepackage{tocloft}
\usepackage{titlesec}
\usepackage{booktabs,array,longtable,enumitem,microtype}
\usepackage{extpfeil}
\usepackage{centernot}

\usepackage[linesnumbered,ruled,vlined]{algorithm2e}
\usepackage{url}

\usepackage[margin=1in]{geometry}

\usepackage[T1]{fontenc}
\usepackage{lmodern}
\usepackage{microtype}
\usepackage{rotating}
\usepackage{multirow}
\usepackage{extarrows}
\usepackage[sort,compress,numbers]{natbib}
\usepackage[utf8]{inputenc}
\usepackage{xcolor}
\usepackage{listings}
\usepackage{aliascnt}
\usepackage[pagebackref,colorlinks=true,linkcolor=blue,citecolor=red,urlcolor=blue]{hyperref}
\usepackage[nameinlink,noabbrev,capitalize]{cleveref}
\numberwithin{equation}{section}

\setlist[enumerate]{label=\textup{(\roman*)},leftmargin=*}
\newtheorem{theorem}{Theorem}[section]
\newaliascnt{proposition}{theorem}
\newtheorem{proposition}[proposition]{Proposition}
\aliascntresetthe{proposition}
\newaliascnt{lemma}{theorem}
\newtheorem{lemma}[lemma]{Lemma}
\aliascntresetthe{lemma}
\newaliascnt{corollary}{theorem}
\newtheorem{corollary}[corollary]{Corollary}
\aliascntresetthe{corollary}

\theoremstyle{definition}
\newaliascnt{definition}{theorem}
\newtheorem{definition}[definition]{Definition}
\aliascntresetthe{definition}
\newaliascnt{remark}{theorem}
\newtheorem{remark}[remark]{Remark}
\aliascntresetthe{remark}
\newaliascnt{example}{theorem}
\newtheorem{example}[example]{Example}
\aliascntresetthe{example}
\newaliascnt{conjecture}{theorem}

\aliascntresetthe{conjecture}
\newaliascnt{question}{theorem}

\aliascntresetthe{question}
\newaliascnt{problem}{theorem}
\newtheorem{problem}[problem]{Problem}
\aliascntresetthe{problem}
\newaliascnt{alg}{theorem}

\aliascntresetthe{alg}

\allowdisplaybreaks[4]

\SetKwInput{KwInput}{Input}                
\SetKwInput{KwOutput}{Output}              

\newcommand{\R}{\mathbb{R}}
\newcommand{\Z}{\mathbb{Z}}
\newcommand{\N}{\mathbb{Z}_{\geq 0}}

\newcommand{\conv}{\operatorname{conv}}

\newcommand{\codeg}{\operatorname{codeg}}
\newcommand{\relint}{\operatorname{relint}}

\newcommand{\A}{\mathsf{A}}

\title{Taylor Positivity of Ehrhart Polynomials}

\author{Feihu Liu$^{\color{blue} \dag}$ and Zihao Zhang$^{\color{blue} \S}$
\\[2mm]
{\small $^{\color{blue} \dag}$ Center for Combinatorics, LPMC,}\\[-0.8ex]
{\small Nankai University, Tianjin 300071, P.R.~China}\\
{\small $^{\color{blue} \S}$ School of Mathematics and Statistics,}\\[-0.8ex]
{\small Beijing Institute of Technology, Beijing 102400, P.R.~China}\\
{\small {\color{blue} $^\dag$} Email address: liufeihu7476@163.com}\\
{\small {\color{blue} $^\S$} Email address: zihao-zhang@foxmail.com}\\
}

\date{\today}

\begin{document}

\maketitle

\begin{abstract}
Let $P$ be a $d$-dimensional lattice polytope with Ehrhart polynomial $L_P(t)$. Motivated by the study of Ehrhart positivity and magic positivity, we investigate the Taylor coefficients $\mathsf{A}_j(P;k)$ in the shifted expansion $L_P(t)=\sum_{j=0}^{d}\mathsf{A}_j(P;k)(t-k)^j$ about a real center $k$. In this paper, we obtain the following four main results.

(i) We give exact formulas for these coefficients in terms of the ordinary Ehrhart coefficients, the $h^*$-vector, elementary symmetric functions, and Stirling numbers.

(ii) We denote by $\tau(P)$ and $\tau^+(P)$ the smallest nonnegative integral centers at which all Taylor coefficients are nonnegative and positive, respectively. If $s$ is the degree of the $h^*$-polynomial, then $0\leq\tau(P)\leq\tau^+(P)\leq\min\{\max\{0,s-1\},\lfloor\frac{d-1}{2}\rfloor\}$. As an application, we slightly improve an upper bound due to Beck, De Loera, Develin, Pfeifle, and Stanley. That is, every real root of $L_P(t)$ lies in $[-d,\lfloor\frac{d-1}{2}\rfloor)$.

(iii) Let $\rho(P)$ be the smallest nonnegative real center such that the Taylor coefficients are nonnegative. If $\lambda_{\mathbb{R}}(f)$ denotes the largest real zero of $f(t)$, with value $-\infty$ when no such zero exists, then
$\rho(P)=\max\{0,\max_{0\leq j<d}\lambda_{\mathbb{R}}\!(L_P^{(j)})\}$.

(iv) We establish structural properties of the Taylor coefficients $\mathsf{A}_j(P;k)$, including derivative interlacing, palindromic reflection symmetries, and Laguerre and Newton inequalities.

As a final note, these results provide a systematic partial answer to an open problem listed on the website of the American Institute of Mathematics.
\end{abstract}

\noindent
\begin{small}
\emph{2020 Mathematics subject classification}: Primary 52B20; Secondary 52B05, 05A15, 26C10.
\end{small}

\noindent
\begin{small}
\emph{Keywords}: Ehrhart polynomial; Ehrhart positive; Lattice polytope; $h^*$-polynomial; Taylor coefficient; Simplex; Interlacing property; Laguerre inequality.
\end{small}

\tableofcontents

\section{Introduction}

\subsection{Basic concepts}

A \emph{lattice polytope} is the convex hull of finitely many points of $\Z^d$.
Throughout this paper, we assume that $d\geq 1$.
Let $P\subset\R^d$ be a full-dimensional lattice polytope.  
For $t\in\N$, its \emph{lattice-point enumerator}
\[ L_P(t)=|tP\cap\Z^d|\]
counts the number of integer points in the $t$-th dilation $tP = \{ t\alpha : \alpha \in P \}$ of $P$.

Ehrhart~\cite{Ehrhart1962} proved that the function $L_P(t)$ is a polynomial in $t$ of degree $d$ with constant term $1$.
The resulting polynomial is called the \emph{Ehrhart polynomial} of $P$.
The \emph{Ehrhart series} of $P$ is $\operatorname{Ehr}_P(z)=\sum_{t\geq0}L_P(t)z^t$.
It has the rational form
\begin{equation}\label{eq:ehrhart-series}
\operatorname{Ehr}_P(z)=\frac{h_P^*(z)}{(1-z)^{d+1}},\qquad h_P^*(z)=\sum_{r=0}^{d}h_r^*z^r.
\end{equation}
The numerator $h_P^*(z)$ is the \emph{$h^*$-polynomial} of $P$, and its coefficient vector is the \emph{$h^*$-vector}. 
Stanley~\cite{Stanley1980} proved that the coefficients of $h^*$-polynomial are nonnegative integers.

Expanding $(1-z)^{-d-1}$ and comparing coefficients in \eqref{eq:ehrhart-series} gives the polynomial identity
\begin{equation}\label{eq:hstar-expansion}
  L_P(t)=\sum_{r=0}^{d}h_r^*\binom{t+d-r}{d}.
\end{equation}

For a lattice polytope $P$, let $\relint(P)$ denote its relative interior. 
The reciprocity theorem of Ehrhart and Macdonald is a classical result in polytope theory. 
For every positive integer $n$, we have 
\begin{align}\label{thm:reciprocity}
L_P(-n)=(-1)^d\bigl|\relint(nP)\cap\Z^d\bigr|.
\end{align} 
In particular, if $d$ is even, then $L_P(-1)\geq0$.
Its full general form is due to Macdonald \cite{Macdonald1971}; see also \cite[Theorem~4.1]{BeckRobins2015}.

The \emph{degree} and \emph{codegree} of $P$ are $\deg(P)=\deg h_P^*(z)=:s$ and $\codeg(P)=d+1-s$, respectively.
Equivalently, $s$ is the largest index for which $h_s^*\neq0$. 
In particular, $\codeg(P)$ is the smallest positive integer $k$ such that $\relint(kP)\cap\Z^d\neq \emptyset$; see \cite[Theorem 4.5]{BeckRobins2015}.

For further background on lattice polytopes, we refer to several excellent books~\cite{BeckRobins2015}, \cite[Chapter 4]{RP.Stanley}, and \cite{Ziegler}.

\subsection{Motivation}

In the ordinary monomial basis, one writes $L_P(t)=c_dt^d+\cdots+ c_1t+ c_0$.
It is well known \cite[Corollary 3.20; Theorem 5.6]{BeckRobins2015} that $c_d$ equals the volume of $P$, $c_{d-1}$ equals half of the boundary volume of $P$ (i.e., half the sum of the relative volumes of its facets), and the constant term is $c_0=1$.
Therefore, these three coefficients are always positive.

The remaining coefficients, however, are more intricate and lack a straightforward geometric interpretation \cite{McMullen77}.
We refer to the coefficients $c_i$ for $1\leq i\leq d-2$ as the \emph{middle Ehrhart coefficients} of $P$.
A lattice polytope $P$ is said to be \emph{Ehrhart positive} if $c_i>0$ for all $i$.
Reeve's tetrahedra provide the classical non-Ehrhart positive examples \cite{Reeve1957}.

The study of Ehrhart positivity has attracted considerable attention in recent years.
Numerous families of polytopes have been established as Ehrhart positive, including lattice path matroids \cite{Ferroni-Morales2026}
(which encompass the hypersimplices~\cite{Ferroni2021Hypersimplices}, minimal matroids~\cite{Ferroni2022Minimal}, Schubert matroids~\cite{Fan-Li}, Catalan matroids~\cite{Chen-Li-Yao}), rank-two matroids~\cite{FerroniJochemkoSchroeter2022}, cross-polytopes~\cite[Section~2]{Liu2019}, the $y$-families of generalized permutohedra~\cite{Postnikov2009} (which encompass the Pitman--Stanley polytopes~\cite{Stanley-Pitman}), and cyclic polytopes~\cite{LiuFuA-M}.

Conversely, non-Ehrhart positive examples have been found among order polytopes~\cite{LiuTsuchiya2019,LiuXinZhang2026}, matroid polytopes~\cite{Ferroni2022Matroids}, and smooth polytopes~\cite{CastilloEtAl2018}. More non-Ehrhart-positive polytopes are constructed in \cite{Hibi-Higashitani-Tsuchiya-Yoshida,LiuTaoXin,LiuTaoXinSign}
See \cite{Liu2019} for a survey of Ehrhart positivity.

A stronger property than Ehrhart positivity is \emph{magic positivity}, defined as follows.
\begin{definition}
A lattice polytope $P$ of dimension $d$ with Ehrhart polynomial $L_P(t)$ is said to be \emph{magic positive} if the polynomial can be expressed in the basis $\{t^i(t+1)^{d-i}\}_{i=0}^d$ with nonnegative coefficients. That is,
\begin{align*}
L_P(t) = \sum_{i=0}^d a_i t^i(t+1)^{d-i},
\end{align*}
where $a_i \ge 0$ for all $0 \le i \le d$.
\end{definition}

Clearly, the magic positivity of a lattice polytope implies the Ehrhart positivity of the polytope.
Another reason for studying the magic positivity of a lattice polytope $P$ is that, by the result of Br\"and\'en \cite{Branden06}, it implies that the $h^*$-polynomial of $P$ is real-rooted.
Some results on the magic positivity of polytopes are found in \cite{Avila26,Ferroni-Morales2026,LiuZhangMagic,Konoike,HillTrinh}.
Many polytopes that are Ehrhart positive but not magic positive can also be found in \cite{Avila26,FerroniHigashitani2024,LiuZhangMagic}.

Motivated by investigations on Ehrhart positive and magic positive, we naturally define the shifted Ehrhart coefficients as follows.

\begin{definition}\label{def:taylor-positive}
Let $P$ be a $d$-dimensional lattice polytope and let $k\in\R$.  The \emph{shifted Ehrhart coefficients at $k$} are the uniquely determined numbers $\A_j(P;k)$ satisfying
\begin{equation}\label{eq:shifted-expansion}
L_P(t)=\sum_{j=0}^{d}\A_j(P;k)(t-k)^j.
\end{equation}
We call the coefficients $\A_j(P;k)$ the \emph{Taylor coefficients} of the polytope $P$. We also refer to $k$ as the expansion \emph{center}. We say that $P$ is \emph{Taylor nonnegative at $k$} if $\A_j(P;k)\geq0$ for every $j$, and \emph{Taylor positive at $k$} if $\A_j(P;k)>0$ for every $j$.
\end{definition}

The terminology records that Taylor positivity at $0$ is precisely Ehrhart positivity.

Another motivation for introducing the relevant concepts in \cref{def:taylor-positive} is the following open problem listed on the website of the American Institute of Mathematics~\cite{AimPl-Website}.

\begin{problem}(\cite[Problem 1.6]{AimPl-Website})\label{Problem-Join-open}
Let $P$ be a lattice polytope. Fix $k\in \mathbb{Z}_{>0}$ and write $L_P(t)$ in the basis $\{ (t-k)^i\}_i$.
For example, if $\deg(h_P^*(z))=s$, then $L_P(t)$ is nonnegative in the basis  $\{ (t-s+1)^i\}_{i=0}^d$.
Investigate what happens when changing $k$ from 0 to 1 for a non-Ehrhart positive polytope (or other values of $k$).  
\end{problem}

\subsection{Main results}

In this paper, we give a systematic study of the Taylor coefficients $\A_j(P;k)$ of the polytope $P$.

The \textbf{first} contribution of this paper consists of exact formulas for $\A_j(P;k)$.
If $e_a$ denotes the elementary symmetric polynomial of degree $a$, then
\[\A_j(P;k)=\frac{L_P^{(j)}(k)}{j!}=\sum_{m=j}^{d}\binom{m}{j}c_mk^{m-j}
 =\frac1{d!}\sum_{r=0}^{d}h_r^*e_{d-j}(k-r+1,\ldots,k-r+d).
\]
Moreover, for $\delta\geq0$, we obtain
\[\A_j(P;k+\delta)=\sum_{\ell=j}^{d}\binom{\ell}{j}\delta^{\ell-j}\A_\ell(P;k).
\]
In particular, simultaneous nonnegativity at one center implies positivity at every larger center.

We now define the following two indices. 

\begin{definition}\label{def:positivity-indices}
The \emph{integer Taylor-nonnegativity index} and the \emph{integer Taylor-positivity index} are
\begin{align*}
 \tau(P)&=\min\{k\in\N:\A_j(P;k)\geq0\text{ for all }j\},\\
 \tau^+(P)&=\min\{k\in\N:\A_j(P;k)>0\text{ for all }j\}.
\end{align*}
\end{definition}
These quantities measure how far the expansion center must be shifted to the right from zero before coefficientwise nonnegativity or positivity is attained.
These minima exist.  Indeed, for $j<d$ the polynomial $k\mapsto\A_j(P;k)$ has degree $d-j$ and positive leading coefficient
$\binom{d}{j}c_d$, while $\A_d(P;k)=c_d>0$.  Thus every coefficient is positive for all sufficiently large $k$.

As a \textbf{second} contribution, we give the following upper bound.
\begin{theorem}\label{Intorodct-One}
Let $P$ be a $d$-dimensional lattice polytope, and let $s=\deg h_P^*(z)$.  Then
\[ 0\leq\tau(P)\leq\tau^+(P)\leq\min\left\{\max\{0,s-1\},\left\lfloor\frac{d-1}{2}\right\rfloor\right\}.
\]
\end{theorem}

By constructing a class of simplices, we also prove the following result.
\begin{corollary}\label{Introduction-Cor-ODD-EVEN}
For every odd $d\geq3$, there is a $d$-dimensional lattice simplex for which $\tau(P)=\tau^+(P)=\frac{d-1}{2}$.
For every even $d\geq4$, there is a $d$-dimensional lattice simplex for which $\tau(P)=\tau^+(P)=\frac{d-2}{2}$.
\end{corollary}

For $d\geq1$, define
\begin{align*}
\mathfrak u_d&:=\min\left\{
\begin{aligned}
k\in\N:\quad&\A_j(P;k)\geq0\text{ for every }d\text{-dimensional lattice polytope }P\\[-2pt]
&\text{and every }0\leq j\leq d
\end{aligned}\right\},
\\
\mathfrak u_d^+&:=\min\left\{
\begin{aligned}
k\in\N:\quad&\A_j(P;k)>0\text{ for every }d\text{-dimensional lattice polytope }P\\[-2pt]
&\text{and every }0\leq j\leq d
\end{aligned}\right\}.
\end{align*}

This yield the following result.
\begin{theorem}\label{Intorodct-Two}
For every $d\geq1$, we have 
\[\mathfrak u_d=\mathfrak u_d^+=\left\lfloor\frac{d-1}{2}\right\rfloor.
\]
\end{theorem}

As an application, we slightly improve an upper bound for Beck, De Loera, Develin, Pfeifle, and Stanley \cite{BeckEtAl2005}.
\begin{corollary}\label{Intorodct-Three}
Let $P$ be a $d$-dimensional lattice polytope.  Then every real root $\alpha$ of $L_P(t)$ satisfies $\alpha<\lfloor\frac{d-1}{2}\rfloor$. Together with the lower bound in
\cite[Theorem~1.2(b)]{BeckEtAl2005}, this shows that every real Ehrhart root lies in
\[\left[-d,\left\lfloor\frac{d-1}{2}\right\rfloor\right).
\]
\end{corollary}

To define the corresponding least real center, set
\[\mathcal N(P)=\{k\in\R_{\geq0}:\A_j(P;k)\geq0\text{ for all }j\}.
\]
This set is nonempty, closed by continuity, and upward closed by \cref{cor:positivity-persistence}.  Hence there is a unique $\rho(P)\geq0$ such that
\begin{equation}\label{eq:real-positivity-ray}
\mathcal N(P)=[\rho(P),\infty).
\end{equation}
We refer to $\rho(P)$ as the \emph{real Taylor-nonnegativity index}.

As our \textbf{third} main contribution, the real zeros of the coefficient functions determine $\rho(P)$.  
If $\lambda_{\R}(f)$ is the largest real zero of $f(t)$, with the convention $\lambda_{\R}(f)=-\infty$ when $f(t)$ has no real zero, then
\[\rho(P)=\max\left\{0,\max_{0\leq j<d}\lambda_{\R}\!\left(L_P^{(j)}\right)\right\}.
\]
Furthermore, we also derive upper and lower bounds for $\rho(P)$, see \cref{thm:rho-root-sandwich}.

The \textbf{fourth} main contribution of this paper is to study various properties of the Taylor coefficients $\A_j(P;k)$, such as the following interlacing property.
\begin{theorem}
Fix $0\leq j< d-1$.
Suppose $\A_j(P;\,\cdot\,)$ has only real zeros.  Then $\A_{j+1}(P;\,\cdot\,)$ has only real zeros and its zeros weakly
interlace those of $\A_j(P;\,\cdot\,)$.  If all zeros of $\A_j(P;\,\cdot\,)$ are simple, the interlacing is strict.
In particular, if $L_P$ is real-rooted and $\alpha_{j,1}\leq\alpha_{j,2}\leq\cdots\leq\alpha_{j,d-j}$ 
are the zeros of $\A_j(P;\,\cdot\,)$, repeated according to multiplicity, then
\[\alpha_{j,1}\leq\alpha_{j+1,1}\leq\alpha_{j,2}\leq\cdots\leq\alpha_{j+1,d-j-1}\leq\alpha_{j,d-j}.
\]
\end{theorem}

We also have the following Laguerre and Newton inequalities.
\begin{theorem}
Fix $0\leq j\leq d-2$.  If the polynomial $k\mapsto\A_j(P;k)$ is real-rooted, then, for every $k\in\R$,
\begin{equation*}
(j+1)\A_{j+1}(P;k)^2\geq(j+2)\A_j(P;k)\A_{j+2}(P;k).
\end{equation*}
If $\A_j(P;\,\cdot\,)$ has only simple zeros, the inequality is strict.
\end{theorem}

\begin{theorem}
Suppose $L_P(t)$ is real-rooted.  Then, for every $k\geq\rho(P)$ and every $1\leq j\leq d-1$, we have 
\begin{equation*}
\left(\frac{\A_j(P;k)}{\binom dj}\right)^2\geq\frac{\A_{j-1}(P;k)}{\binom d{j-1}}\frac{\A_{j+1}(P;k)}{\binom d{j+1}}.
\end{equation*}
Equivalently,
\begin{equation*}
\A_j(P;k)^2\geq\frac{(j+1)(d-j+1)}{j(d-j)}\A_{j-1}(P;k)\A_{j+1}(P;k).
\end{equation*}
\end{theorem}

In addition, we also investigate products of polytopes, dilations, and palindromicity; see \cref{Subsection-PDP-prob}.
We also determine the least centers for several natural constructions and families, including order polytopes, matroid base polytopes, generalized permutohedra, and smooth polytopes.

The paper is organized as follows. 
In \cref{sec:exact}, we give the exact formulas for the Taylor coefficients $\A_{j}(P;k)$.
General bounds for the center $k$ are proved in \cref{sec:bounds}. 
In \cref{sec:unit-center}, we construct a family of generalized Reeve simplices and we use them to prove \cref{Introduction-Cor-ODD-EVEN}.
\cref{sec:optimal-universal-shift} is devoted to the proofs of \cref{Intorodct-One}, \cref{Intorodct-Two}, and \cref{Intorodct-Three}.
In \cref{sec:root-geometry}, we focus mainly on real and integer Taylor-nonnegativity indices.
In \cref{Section-Prop-Taylor}, we mainly study various properties of the Taylor coefficients  $\A_j(P;k)$, including interlacing property, Laguerre and Newton inequalities.
Some natural classes are treated in \cref{sec:natural-classes}.  
Finally, \cref{sec:conclusion} records the remaining questions.

\section{Exact formulas for Taylor coefficients}
\label{sec:exact}

Recall that a smooth real-valued function is \emph{absolutely monotone} on an interval if all of its derivatives are nonnegative there. For polynomials, Taylor nonnegativity is exactly the onset of absolute monotonicity.

\begin{proposition}\label{prop:absolute-monotonicity}
Let $f\in\R[t]$ have degree at most $d$, and write
\[
 f(k+x)=\sum_{j=0}^{d}a_jx^j.
\]
Then $a_j\geq0$ for every $j$ if and only if $f^{(m)}(t)\geq0$ for every $t\in[k,\infty)$ and every $m\in\N$.
In particular, for a lattice polytope $P$, Taylor nonnegativity at $k$ is
equivalent to absolute monotonicity of $L_P(t)$ on $[k,\infty)$.
\end{proposition}
\begin{proof}
For $0\leq m\leq d$ and $x\geq0$, termwise differentiation gives
\[f^{(m)}(k+x)=\sum_{j=m}^{d}\frac{j!}{(j-m)!}a_jx^{j-m}.
\]
Assume first that every $a_j$ is nonnegative. We have $f^{(m)}(k+x)\geq0$ because every factor in every summand is nonnegative.  Derivatives of order $m>d$ vanish identically and are therefore nonnegative as well.
Conversely, suppose every derivative of $f$ is nonnegative on
$[k,\infty)$.  Evaluating at the left endpoint and using the uniqueness of
the Taylor expansion gives $a_j=\frac{f^{(j)}(k)}{j!}\geq0$ ($0\leq j\leq d$).
This proves both implications.
\end{proof}

For a finite list $a_1,\ldots,a_d$, let $e_n(a_1,\ldots,a_d)$ denote the \emph{elementary symmetric polynomial} of degree $n$, with $e_0=1$ and
$$\prod_{i=1}^d (1+a_iy)=\sum_{n=0}^{d}e_n(a_1,\ldots,a_d)y^n.$$ 
For more on $e_n(a_1,\ldots,a_d)$, we refer to \cite[Chapter 7]{RP.Stanley99}.
Let $s(d,\ell)$ denote an unsigned Stirling number of the first kind (see \cite[Chapter 1]{RP.Stanley}), characterized by
\begin{equation}\label{eq:stirling-def}
y(y+1)\cdots(y+d-1) =\sum_{\ell=0}^{d}s(d,\ell)y^\ell.
\end{equation}

\begin{theorem}\label{thm:exact-formulas}
Assume that the Ehrhart polynomial and the $h^*$-polynomial of the polytope $P$ are, respectively, 
\[L_P(t)=\sum_{m=0}^{d}c_mt^m,\qquad h_P^*(z)=\sum_{r=0}^{d}h_r^*z^r.
\]
For every $k\in\R$ and every $0\leq j\leq d$, we have 
\begin{align}
 \A_j(P;k) &=\frac{L_P^{(j)}(k)}{j!}, \label{eq:formula-derivative}\\
 &=\sum_{m=j}^{d}\binom{m}{j}c_mk^{m-j}, \label{eq:formula-monomial}\\
 &=\frac1{d!}\sum_{r=0}^{d}h_r^* e_{d-j}(k-r+1,k-r+2,\ldots,k-r+d), \label{eq:formula-elementary}\\
 &=\frac1{d!}\sum_{r=0}^{d}h_r^* \sum_{\ell=j}^{d} s(d,\ell)\binom{\ell}{j} (k-r+1)^{\ell-j}. \label{eq:formula-stirling}
\end{align}
Terms with $r>\deg h_P^*$ may of course be omitted.
\end{theorem}
\begin{proof}
Set $x=t-k$.  Taylor's formula is an exact finite identity for the
polynomial $L_P$:
\[
  L_P(k+x)=\sum_{j=0}^{d}\frac{L_P^{(j)}(k)}{j!}x^j.
\]
Comparing this with \eqref{eq:shifted-expansion} gives \eqref{eq:formula-derivative}.
Next, substitute $t=k+x$ into the ordinary monomial expansion and use the
binomial theorem:
\begin{align*}
L_P(k+x)=\sum_{m=0}^{d}c_m(k+x)^m=\sum_{m=0}^{d}c_m\sum_{j=0}^{m}\binom{m}{j}k^{m-j}x^j
=\sum_{j=0}^{d}\left(\sum_{m=j}^{d}\binom{m}{j}c_mk^{m-j}\right)x^j.
\end{align*}
Coefficient comparison proves \eqref{eq:formula-monomial}.
For the $h^*$-formula, use \eqref{eq:hstar-expansion}:
\begin{align}\label{eq:factorized-shift}
L_P(k+x)=\sum_{r=0}^{d}h_r^*\binom{x+k+d-r}{d}=\frac1{d!}\sum_{r=0}^{d}h_r^*\prod_{\nu=1}^{d}(x+k-r+\nu).
\end{align}
For arbitrary constants $a_1,\ldots,a_d$,
\[\prod_{\nu=1}^{d}(x+a_\nu)=\sum_{j=0}^{d}e_{d-j}(a_1,\ldots,a_d)x^j.
\]
Apply this identity to $a_\nu=k-r+\nu$ in \eqref{eq:factorized-shift}; the coefficient of $x^j$ is exactly the right-hand side of \eqref{eq:formula-elementary}.
Finally, set $y=x+k-r+1$.  The product associated with the index $r$ is the rising factorial $y(y+1)\cdots(y+d-1)$. By \eqref{eq:stirling-def}, it equals
\[\sum_{\ell=0}^{d}s(d,\ell)(x+k-r+1)^\ell.
\]
Expanding the final power by the binomial theorem, the coefficient of $x^j$ is
\[\sum_{\ell=j}^{d}s(d,\ell)\binom{\ell}{j}(k-r+1)^{\ell-j}.
\]
Summing over $r$ and dividing by $d!$ proves \eqref{eq:formula-stirling}.
\end{proof}

The coefficient vectors at different centers are related by binomial coefficients.

\begin{theorem}\label{thm:translation-semigroup}
For all real $k,\delta$ and every $0\leq j\leq d$, we have
\begin{equation}\label{eq:translation-semigroup}
 \A_j(P;k+\delta)=\sum_{\ell=j}^{d}\binom{\ell}{j}\delta^{\ell-j}\A_\ell(P;k).
\end{equation}
Moreover,
\begin{equation}\label{eq:coefficient-ode}
\frac{d}{dk}\A_j(P;k)=(j+1)\A_{j+1}(P;k)\qquad(0\leq j<d),
\end{equation}
and $\A_d(P;k)=c_d$ is constant.
\end{theorem}

\begin{proof}
Set $y=t-(k+\delta)$.  Then $t-k=y+\delta$.  Starting from the expansion
at $k$ gives
\begin{align*}
 L_P(t)=\sum_{\ell=0}^{d}\A_\ell(P;k)(y+\delta)^\ell
 =\sum_{\ell=0}^{d}\A_\ell(P;k)\sum_{j=0}^{\ell}\binom{\ell}{j}\delta^{\ell-j}y^j
 =\sum_{j=0}^{d}\left(\sum_{\ell=j}^{d}\binom{\ell}{j}\delta^{\ell-j}\A_\ell(P;k)\right)y^j.
\end{align*}
The uniqueness of the expansion in powers of $y=t-(k+\delta)$ proves \eqref{eq:translation-semigroup}.
Equation \eqref{eq:formula-derivative} gives
\[\frac{d}{dk}\A_j(P;k)=\frac{L_P^{(j+1)}(k)}{j!}=(j+1)\A_{j+1}(P;k),
\]
which is \eqref{eq:coefficient-ode}.  The assertion about $\A_d$ follows either from the same formula or from \eqref{eq:formula-monomial}.
\end{proof}

\begin{corollary}\label{cor:positivity-persistence}
Suppose $\A_j(P;k_0)\geq0$ for all $j$. Then $\A_j(P;k)\geq0$ for every $j$ and every $k\geq k_0$. 
If $k>k_0$, all these coefficients are positive.
\end{corollary}
\begin{proof}
Write $k=k_0+\delta$ with $\delta\geq0$ and apply \eqref{eq:translation-semigroup}.  Every summand is nonnegative.  If
$\delta>0$ and $j<d$, the term with $\ell=d$ is
\[\binom{d}{j}\delta^{d-j}\A_d(P;k_0)=\binom{d}{j}\delta^{d-j}c_d>0.
\]
For $j=d$, the coefficient is $c_d>0$ directly.
\end{proof}

At an integral center, three coefficients are automatically positive.

\begin{proposition}\label{prop:automatic-coefficients}
For every integer $k\geq0$, we have 
\[\A_0(P;k)>0,\qquad\A_{d-1}(P;k)>0,\qquad\A_d(P;k)>0.
\]
Hence only $\A_1(P;k),\ldots,\A_{d-2}(P;k)$ can be negative.
\end{proposition}
\begin{proof}
By the defining counting property,
$\A_0(P;k)=L_P(k)=|kP\cap\Z^d|>0$.  Formula \eqref{eq:formula-monomial} gives
\[  \A_d(P;k)=c_d>0, \qquad \A_{d-1}(P;k)=c_{d-1}+dkc_d>0.
\]
This completes the proof.
\end{proof}

\begin{proposition}\label{prop:index-gap}
For every lattice polytope $P$, we have 
\[\tau(P)\leq\tau^+(P)\leq\tau(P)+1.
\]
\end{proposition}
\begin{proof}
The first inequality follows immediately from the definitions.  By definition of $\tau(P)$, every coefficient at the integer center
$k_0=\tau(P)$ is nonnegative.  Apply \cref{cor:positivity-persistence} with the positive displacement
$\delta=1$.  Every coefficient at $k_0+1$ is then positive, so $\tau^+(P)\leq k_0+1=\tau(P)+1$.
\end{proof}

\begin{corollary}\label{cor:root-free}
If $P$ is Taylor nonnegative at $k\geq0$, then $L_P(k)\geq0$ and $L_P(t)>0$ for every real $t>k$. 
Thus $L_P(t)$ has no real root in $(k,\infty)$.  
If, in addition, $L_P(k)>0$ (in particular, if $k$ is a nonnegative integer)
then $L_P(t)$ has no real root in $[k,\infty)$.
\end{corollary}
\begin{proof}
Write $t=k+x$ with $x\geq0$.  Every term in $L_P(k+x)=\sum_{j=0}^{d}\A_j(P;k)x^j$
is nonnegative.  At $x=0$ this gives $L_P(k)=\A_0(P;k)\geq0$.  If $x>0$,
the leading-degree summand is
\[\A_d(P;k)x^d=c_dx^d>0,
\]
so the entire sum is positive.  Finally, if $k\in\N$, then
$L_P(k)=|kP\cap\Z^d|>0$; the same conclusion holds whenever positivity at
the endpoint is assumed directly.
\end{proof}

\section{Bounds for the center}
\label{sec:bounds}

We first prove the degree bound directly from the factorization in \eqref{eq:factorized-shift}.

\begin{theorem}\label{thm:degree-bound}
Let $P$ be a $d$-dimensional lattice polytope and let $s=\deg h_P^*(z)$.  If
\[k\geq\max\{0,s-1\},
\]
then $P$ is Taylor positive at $k$; that is, $\A_j(P;k)>0$ ($0\leq j\leq d$).
The conclusion holds for every real $k$ satisfying this inequality, not only for integral $k$.
\end{theorem}
\begin{proof}
Because $h_r^*=0$ for $r>s$, equation \eqref{eq:factorized-shift} becomes
\begin{equation}\label{eq:degree-bound-factorization}
 L_P(k+x) =\frac1{d!}\sum_{r=0}^{s}h_r^*\prod_{\nu=1}^{d}(x+k-r+\nu).
\end{equation}
Fix $r\leq s$. The smallest constant appearing in its product is $k-r+1\geq k-s+1\geq0$.
It follows that every coefficient of $\prod_{\nu=1}^{d}(x+k-r+\nu)$ is nonnegative.
Each weight $h_r^*$ in \eqref{eq:degree-bound-factorization} is also nonnegative.

It remains to prove positivity rather than mere nonnegativity.
Consider the $r=0$ summand.  Its weight is $h_0^*=1$, and its constants are $k+1,k+2,\ldots,k+d$.
Since $k\geq0$, all of these numbers are positive.  Every coefficient of $\prod_{\nu=1}^{d}(x+k+\nu)$ is positive.  Adding the remaining nonnegative summands preserves positivity in every degree.
\end{proof}

The theorem is sometimes more naturally written in terms of codegree:
since $s-1=d-\codeg(P)$, the degree bound is
$k\geq\max\{0,d-\codeg(P)\}$.

We next record the independent dimension bound hidden in the proof of the real-root theorem of Beck, De Loera, Develin, Pfeifle, and Stanley \cite{BeckEtAl2005}.  To make the deduction transparent, we state precisely the ingredient they prove.
For $0\leq r,\ell\leq d$ and $n\in\R$, define
\begin{equation}\label{eq:g-definition}
g_r(n,\ell)=\sum_{\substack{I\subseteq\{0,1,\ldots,d-1\}\\|I|=d-\ell}}\prod_{a\in I}(n+d-r-a).
\end{equation}

\begin{lemma}{\em (Beck--De Loera--Develin--Pfeifle--Stanley, \cite[Lemmas~4.5--4.6]{BeckEtAl2005})}\label{lem:beck}
Let $B=\lfloor d/2\rfloor$.  For each
$0\leq\ell<d$ there is a
constant $\lambda_\ell>0$ such that
\begin{equation}\label{eq:beck-inequality}
  g_r(B,\ell)\geq\lambda_\ell(d+1-2r)\qquad(0\leq r\leq d).
\end{equation}
More explicitly, with the notation in \eqref{eq:g-definition}, one may
take
\begin{align}
 \lambda_\ell=\frac12\bigl(g_B(B,\ell)-g_{B+1}(B,\ell)\bigr)
 =\frac d2\sum_{\substack{I\subseteq\{1,\ldots,d-1\}\\|I|=d-\ell-1}}\prod_{a\in I}(d-a)>0.
\end{align}
\end{lemma}

The proof of \cref{lem:beck} studies the piecewise-linear interpolation of $r\mapsto g_r(B,\ell)$.  On the initial range it uses positivity and a discrete convexity property; on the remaining range it pairs terms in \eqref{eq:g-definition} so that the positive products dominate the negative ones. 

\begin{proposition}\label{prop:dimension-bound}
Let $B=\lfloor d/2\rfloor$.  Then $\A_\ell(P;B)>0$ ($0\leq\ell\leq d$).
Equivalently, every $d$-dimensional lattice polytope is Taylor positive at $\lfloor d/2\rfloor$.
\end{proposition}
\begin{proof}
Beck et al. \cite[Page 15]{BeckEtAl2005} derived the following formula:
\begin{equation}\label{eq:derivative-g}
 L_P^{(\ell)}(B)=\frac{\ell!}{d!}\sum_{r=0}^{d}h_r^*g_r(B,\ell).
\end{equation}
We next express $c_{d-1}$ in terms of the $h^*$-vector.  The sum of the
constants in $\prod_{\nu=1}^{d}(t+\nu-r)$ is
\[\sum_{\nu=1}^{d}(\nu-r)=\frac{d(d+1)}2-dr=\frac d2(d+1-2r).
\]
Thus the coefficient of $t^{d-1}$ in $\binom{t+d-r}{d}$ equals $\frac{d+1-2r}{2(d-1)!}$.
It follows that
\begin{equation}\label{eq:surface-hstar}
 2(d-1)!c_{d-1}=\sum_{r=0}^{d}h_r^*(d+1-2r)>0,
\end{equation}
where the inequality uses $c_{d-1}>0$.

For $0\leq\ell<d$, choose $\lambda_\ell$ from \cref{lem:beck}.  Since
every $h_r^*$ is nonnegative, \eqref{eq:beck-inequality} and
\eqref{eq:surface-hstar} imply
\begin{align*}
\sum_{r=0}^{d}h_r^*g_r(B,\ell)\geq\lambda_\ell\sum_{r=0}^{d}h_r^*(d+1-2r)=2\lambda_\ell(d-1)!c_{d-1}>0.
\end{align*}
Equation \eqref{eq:derivative-g} now gives $L_P^{(\ell)}(B)>0$ for $0\leq\ell<d$.  For the remaining derivative,
$L_P^{(d)}(B)=d!c_d>0$.  Finally, \eqref{eq:formula-derivative} gives $\A_\ell(P;B)>0$ for every
$0\leq\ell\leq d$.
\end{proof}

Combining \cref{prop:index-gap}, \cref{thm:degree-bound}, and \cref{prop:dimension-bound} gives the following bound.

\begin{corollary}\label{cor:combined-bound}
For every $d$-dimensional lattice polytope of $h^*$-degree $s$, we have 
\[0\leq\tau(P)\leq\tau^+(P)\leq\min\!\left\{\max\{0,s-1\},\left\lfloor\frac d2\right\rfloor\right\}.
\]
\end{corollary}

\section{A sharp family of odd-dimensional simplices}
\label{sec:unit-center}

The center $1$ is the first nontrivial integral translation.  We begin with
an elementary low-dimensional result, then construct a family that is simultaneously useful for the sharpness of the universal estimates.

\subsection{Dimensions at most four}
\label{sec:low-dimensional}

The next theorem includes a self-contained proof.  
Its four-dimensional argument is the coefficient computation underlying \cite[Proposition~4.7]{BeckEtAl2005}.

\begin{theorem}\label{thm:dimension-four}
If $P$ is a lattice polytope of dimension $d\leq4$, then $P$ is Taylor
positive at $1$.
\end{theorem}
\begin{proof}
The assertion is immediate in dimension one: a lattice segment has
$L_P(t)=at+1$ with $a>0$, and $L_P(1+x)=(a+1)+ax$.
In dimension two, write
$L_P(t)=at^2+bt+1$.  Here $a>0$, and $b$ is one half of the
normalized boundary length, so $b>0$.  Hence
\[L_P(1+x)=a x^2+(2a+b)x+(a+b+1),
\]
whose coefficients are all positive.
In dimension three, write $L_P(t)=pt^3+qt^2+rt+1$.
The two highest coefficients satisfy $p>0$ and $q>0$.  A full-dimensional lattice $3$-polytope has at least four lattice points, namely at least four affinely independent vertices.  
Therefore
\[ L_P(1)=p+q+r+1\geq4,\qquad\text{so}\qquad r\geq3-p-q.
\]
Expanding at $1$ gives
\[L_P(1+x)=p x^3+(3p+q)x^2+(3p+2q+r)x+L_P(1).
\]
The only coefficient whose positivity is not immediate satisfies
\[ 3p+2q+r\geq 3p+2q+(3-p-q)=2p+q+3>0.
\]
Finally, let $d=4$ and write $L_P(t)=pt^4+qt^3+rt^2+st+1$, where $p>0$ and $q>0$.  Since $P$ has at least five lattice points,
\begin{equation}\label{eq:d4-one}
p+q+r+s+1=L_P(1)\geq5, \qquad\text{hence}\qquad p+q+r+s\geq4.
\end{equation}
By \eqref{thm:reciprocity}, we have 
\begin{equation}\label{eq:d4-minus-one}
 p-q+r-s+1=L_P(-1)=|\relint(P)\cap\Z^4|\geq0.
\end{equation}
Adding \eqref{eq:d4-one} and \eqref{eq:d4-minus-one} gives $2p+2r+1\geq4$, so $r\geq\frac32-p$.
Now
\begin{align*}
 L_P(1+x)=p x^4+(4p+q)x^3+(6p+3q+r)x^2+(4p+3q+2r+s)x+L_P(1).
\end{align*}
The quadratic coefficient satisfies
\[6p+3q+r \geq 6p+3q+\frac32-p >0.
\]
For the linear coefficient, split it into two expressions already bounded:
\begin{align*}
4p+3q+2r+s=(p+q+r+s)+(3p+2q+r)\geq4+3p+2q+\left(\frac32-p\right)>0.
\end{align*}
The leading, cubic, and constant coefficients are visibly positive.  This completes all four dimensions.
\end{proof}

\subsection{Generalized Reeve simplices}
\label{sec:sharp-family}

We now introduce a special class of simplices which is studied in \cite{LiuTaoXinSign}.  
It belongs to the class of empty odd-dimensional simplices described by the generalized
White theorem \cite[Theorem~1.10]{BatyrevHofscheier2021}.

Fix integers $s\geq2$ and $q\geq1$, and set $d=2s-1$.  Let
\[ v_{s,q}=(1,q-1,1,q-1,\ldots,1,q-1,q)\in\Z^{2s-1},
\]
where $(1,q-1)$ occurs $s-1$ times.  Define
\begin{equation*}
 \Delta_{s,q} =\conv(0,e_1,\ldots,e_{2s-2},v_{s,q})\subset\R^{2s-1}.
\end{equation*}
The determinant of the matrix with columns $e_1,\ldots,e_{2s-2},v_{s,q}$ is $q$.  Thus $\Delta_{s,q}$ is
full-dimensional and has normalized volume $q$; equivalently, its Ehrhart leading coefficient is $q/(2s-1)!$.

\begin{proposition}{\em (Liu--Tao--Xin, \cite{LiuTaoXinSign})}\label{prop:family-hstar}
The $h^*$-polynomial of $\Delta_{s,q}$ is
\begin{equation*}
  h_{\Delta_{s,q}}^*(z)=1+(q-1)z^s.
\end{equation*}
Consequently,
\begin{equation}\label{eq:family-ehrhart}
 L_{\Delta_{s,q}}(t)
 =\binom{t+2s-1}{2s-1}
 +(q-1)\binom{t+s-1}{2s-1}.
\end{equation}
\end{proposition}

The next theorem proves optimality of the bound in \cref{cor:combined-bound}.

\begin{theorem}\label{thm:sharp-center}
Let $s\geq2$ and set $H_0=0$ and $H_m=\sum_{a=1}^{m}a^{-1}$ for $m\geq1$.
If
\begin{equation}\label{eq:q-large-condition}
 q>1+(2s-2)(2s-1)\binom{3s-3}{2s-1}\bigl(H_{3s-3}-H_{s-2}\bigr),
\end{equation}
then $\tau(\Delta_{s,q})=\tau^+(\Delta_{s,q})=s-1$.
\end{theorem}
\begin{proof}
By \cref{prop:family-hstar}, the $h^*$-degree of $\Delta_{s,q}$ is $s$.
Hence \cref{thm:degree-bound} shows that every shifted coefficient at $k=s-1$ is positive.  It remains to show failure at $k=s-2$.

Write $t=(s-2)+x$.  The first summand of \eqref{eq:family-ehrhart} becomes $F_s(x)=\binom{x+3s-3}{2s-1}$.
Its linear coefficient is its derivative at zero.  
From $F_s(x)=\frac1{(2s-1)!}\prod_{a=s-1}^{3s-3}(x+a)$, 
the logarithmic derivative, or a direct product-rule computation, gives
\begin{align*}
[x]F_s(x)=F_s(0)\sum_{a=s-1}^{3s-3}\frac1a =\binom{3s-3}{2s-1}\bigl(H_{3s-3}-H_{s-2}\bigr).
\end{align*}
The second summand without its factor $q-1$ becomes
\begin{align*}
 G_s(x)=\binom{x+2s-3}{2s-1}=\frac{(x-1)x(x+1)(x+2)\cdots(x+2s-3)}{(2s-1)!}.
\end{align*}
Thus we have 
\begin{align*}
 [x]G_s(x)=\frac{(-1)\cdot1\cdot2\cdots(2s-3)}{(2s-1)!}=-\frac{1}{(2s-2)(2s-1)}.
\end{align*}
Therefore
\begin{equation}\label{eq:sharp-linear-coefficient}
 \A_1(\Delta_{s,q};s-2)=\binom{3s-3}{2s-1}\bigl(H_{3s-3}-H_{s-2}\bigr)-\frac{q-1}{(2s-2)(2s-1)}.
\end{equation}
Condition \eqref{eq:q-large-condition} makes this quantity negative.
Thus the simplex is not Taylor nonnegative at $s-2$.

If it had been Taylor nonnegative at some earlier nonnegative integer
$k\leq s-3$, \cref{cor:positivity-persistence} would make it Taylor
positive at $s-2$, contradicting
\eqref{eq:sharp-linear-coefficient}.  Hence neither $\tau(\Delta_{s,q})$
nor $\tau^+(\Delta_{s,q})$ is smaller
than $s-1$, while \cref{thm:degree-bound} shows that both are at most
$s-1$. This completes the proof.
\end{proof}

The preceding simplices also produce sharp examples in every larger
dimension.  We give the construction and the coefficient calculation in
full.

If $P\subset\R^m$ is a lattice polytope, its \emph{standard lattice pyramid} is
\[\operatorname{pyr}(P)=\conv\bigl((P,0),e_{m+1}\bigr)\subset\R^{m+1}.\]
For $a\in\N$, the notation $\operatorname{pyr}^{a}(P)$ means the result of applying this construction $a$ times, with
$\operatorname{pyr}^{0}(P)=P$.
For every lattice polytope $P$, the $h^*$-polynomial of $\operatorname{pyr}(P)$ (see \cite[Theorem 2.4]{BeckRobins2015}) is given by
$$h_{\operatorname{pyr}(P)}^*(z)=h_P^*(z).$$

For $D\geq2s-1$, define
\begin{equation*}
 \Delta_{s,q}^{(D)} :=\operatorname{pyr}^{D-(2s-1)}(\Delta_{s,q}).
\end{equation*}
\begin{theorem}\label{thm:pyramid-sharpness}
Let $s\geq2$, $D\geq2s-1$, and $q\geq2$.  Then we have $h_{\Delta_{s,q}^{(D)}}^*(z)=1+(q-1)z^s$ and 
\begin{align}
L_{\Delta_{s,q}^{(D)}}(t)&=\binom{t+D}{D}+(q-1)\binom{t+D-s}{D}.\label{eq:pyramid-ehrhart}
\end{align}
Set
\begin{equation*}
 C_{D,s}:= \binom{D+s-2}{D}\bigl(H_{D+s-2}-H_{s-2}\bigr).
\end{equation*}
If $q>1+D(D-1)C_{D,s}$, then we have $\tau(\Delta_{s,q}^{(D)})=\tau^+(\Delta_{s,q}^{(D)})=s-1$.
\end{theorem}
\begin{proof}
Repeated application of $h_{\operatorname{pyr}(P)}^*(z)=h_P^*(z)$, followed by
\cref{prop:family-hstar}, proves $h_{\Delta_{s,q}^{(D)}}^*(z)=1+(q-1)z^s$.  Since the
dimension is $D$, substituting this $h^*$-polynomial into
\eqref{eq:hstar-expansion} gives \eqref{eq:pyramid-ehrhart}.
We calculate the linear shifted coefficient at $k=s-2$.  Write $t=s-2+x$.  The first term on the right-hand side of
\eqref{eq:pyramid-ehrhart} is
\[\binom{x+D+s-2}{D}=\frac1{D!}\prod_{a=s-1}^{D+s-2}(x+a).
\]
All constants in this product are positive.  Differentiating at zero by the product rule yields
\begin{align*}
 [x]\binom{x+D+s-2}{D}&=\frac1{D!}\sum_{b=s-1}^{D+s-2}\prod_{\substack{a=s-1\\a\ne b}}^{D+s-2}a
 =\frac1{D!}\left(\prod_{a=s-1}^{D+s-2}a\right)\sum_{b=s-1}^{D+s-2}\frac1b\\
 &=\binom{D+s-2}{D}\bigl(H_{D+s-2}-H_{s-2}\bigr)=C_{D,s}.
\end{align*}
We also have 
\begin{align*}
[x]\binom{x+D-2}{D}=\frac{(-1)\,1\cdot2\cdots(D-2)}{D!}=-\frac1{D(D-1)}.
\end{align*}
Combining the two calculations gives the exact identity
\begin{equation}\label{eq:pyramid-linear}
 \A_1(\Delta_{s,q}^{(D)};s-2)=C_{D,s}-\frac{q-1}{D(D-1)}.
\end{equation}
Under $q>1+D(D-1)C_{D,s}$, this coefficient is negative.

On the other hand, $h_{\Delta_{s,q}^{(D)}}^*(z)=1+(q-1)z^s$ has degree $s$, so \cref{thm:degree-bound} gives positivity at $k=s-1$. 
The negative coefficient in \eqref{eq:pyramid-linear} rules out $k=s-2$. If Taylor nonnegativity held at an integer $k\leq s-3$, then \cref{cor:positivity-persistence} would imply Taylor positivity at $s-2$, again a contradiction.
Thus neither $\tau$ nor $\tau^+$ is smaller than $s-1$, and positivity there proves that neither is larger.
\end{proof}

\begin{remark}\label{rem:real-sharpness}
The equality in \cref{thm:sharp-center,thm:pyramid-sharpness} concern the integral invariants in \cref{def:positivity-indices}.  Under the corresponding large-$q$ hypothesis, the least real center in \eqref{eq:real-positivity-ray} instead satisfies
\[ s-2<\rho(\Delta_{s,q}^{(D)})<s-1.
\]
The first inequality follows from the negative coefficient at $s-2$.  At $s-1$ every coefficient is positive, so continuity gives positivity throughout a sufficiently small interval immediately to the left of $s-1$, proving the second inequality.  Thus an integer sharpness statement must not be read as an equality for real expansion centers.
\end{remark}

The theorems have the following immediate consequences.

\begin{corollary}\label{cor:dimension-sharp}
For every odd $d\geq3$, there is a $d$-dimensional lattice simplex for which $\tau(P)=\tau^+(P)=\lfloor d/2\rfloor$.  Hence the universal dimension bound in \cref{prop:dimension-bound} is sharp in every odd dimension.
\end{corollary}
\begin{proof}
Write $d=2s-1$ and choose $q$ satisfying
\eqref{eq:q-large-condition}.  Then \cref{thm:sharp-center} gives $\tau(\Delta_{s,q})=\tau^+(\Delta_{s,q})=s-1=\lfloor\frac d2\rfloor$.
\end{proof}

\begin{corollary}\label{cor:fixed-k}
Fix $k\in\N$. The following result holds.
\begin{enumerate}
 \item[(a)] Every lattice polytope with $\deg h_P^*\leq k+1$ is Taylor positive
 at $k$.
 \item[(b)] There is a $(2k+3)$-dimensional lattice simplex with
 $\deg h_P^*=k+2$ which is not Taylor nonnegative at $k$.
\end{enumerate}
Consequently, no fixed nonnegative integer shift works for lattice polytopes in all dimensions.
\end{corollary}
\begin{proof}
Part (a) is \cref{thm:degree-bound}, because $\deg h_P^*-1\leq k$.  For part (b), take $s=k+2$ and choose $q$ as in
\eqref{eq:q-large-condition}.  Then $k=s-2$, and \eqref{eq:sharp-linear-coefficient} is negative.  The dimension of the
simplex is $2s-1=2k+3$.
\end{proof}

\subsection{From center zero to center one}
\label{sec:shift-one}

At $k=0$, \eqref{eq:formula-monomial} gives
$\A_j(P;0)=c_j$.  At $k=1$, it gives the finite Pascal transform
\begin{equation}\label{eq:pascal-transform}
  \A_j(P;1)=\sum_{m=j}^{d}\binom{m}{j}c_m.
\end{equation}
Thus a negative ordinary Ehrhart coefficient may be outweighed by higher
coefficients, but \eqref{eq:pascal-transform} does not preserve positivity
when some input coefficients are negative.

The degree classification begins as follows.

\begin{proposition}\label{prop:degree-one-two}
Let $s=\deg h_P^*$. The following result holds.
\begin{enumerate}
 \item[(a)] If $s\leq1$, then $P$ is Ehrhart positive.
 \item[(b)] If $s=2$, then $P$ is Taylor positive at $1$, although it need not be
 Ehrhart positive.
 \item[(c)] Already among polytopes with $s=3$, Taylor positivity at $1$ is not
 determined by the $h^*$-degree alone.
\end{enumerate}
\end{proposition}
\begin{proof}
Parts (a) and (b) follow from \cref{thm:degree-bound} at $k=0$ and $k=1$,
respectively.  Part (c) will be demonstrated by the two members
$\Delta_{3,19}$ and $\Delta_{3,118}$ below, which have the same $h^*$-degree
but opposite behavior at $1$.
\end{proof}

For $s=2$, the family in \cref{sec:sharp-family} recovers the standard
Reeve phenomenon.  From \eqref{eq:family-ehrhart}, we have 
\[ L_{\Delta_{2,q}}(t)=\binom{t+3}{3}+(q-1)\binom{t+1}{3}=1+\frac{12-q}{6}t+t^2+\frac q6t^3.
\]
It is Ehrhart positive exactly for $1\leq q\leq11$, Taylor nonnegative at
$0$ exactly for $1\leq q\leq12$, and has a negative coefficient exactly
for $q\geq13$.  Nevertheless,
\begin{equation*}
 L_{\Delta_{2,q}}(1+x)=4+\frac{q+12}{3}x+\frac{q+2}{2}x^2+\frac q6x^3
\end{equation*}
has only positive coefficients.

The next result gives the complete five-dimensional calculation.

\begin{example}\label{prop:d5-family}
By \cref{prop:family-hstar}, for $q\geq2$ we have 
\begin{align*}
 L_{\Delta_{3,q}}(t)={}&1+\frac{2q+135}{60}t+\frac{15}{8}t^2+\frac{18-q}{24}t^3+\frac18t^4+\frac q{120}t^5,\\
 L_{\Delta_{3,q}}(1+x)={}&6+\frac{175-q}{20}x+\frac{117-q}{24}x^2+\frac{q+30}{24}x^3+\frac{q+3}{24}x^4+\frac q{120}x^5,\\
 L_{\Delta_{3,q}}(2+x)={}&21+\frac{4q+455}{20}x+\frac{5(2q+45)}{24}x^2+\frac{7(q+6)}{24}x^3+\frac{2q+3}{24}x^4+\frac q{120}x^5.
\end{align*}
Consequently:
\begin{enumerate}
 \item at $0$, it is Taylor positive for $2\leq q\leq17$, Taylor nonnegative but not Taylor positive for $q=18$, and has a negative
 coefficient for every $q\geq19$;
 \item it is Taylor positive at $1$ for $2\leq q\leq116$;
 \item it is Taylor nonnegative but not Taylor positive at $1$ when $q=117$;
 \item it is not Taylor nonnegative at $1$ for every $q\geq118$;
 \item it is Taylor positive at $2$ for every $q\geq2$.
\end{enumerate}
\end{example}

Dimension five is the smallest dimension in which a lattice polytope can fail to be Taylor nonnegative at $1$.  More precisely,      
\cref{thm:dimension-four} shows that every lattice polytope of dimension at most four is Taylor positive at $1$, whereas 
\cref{prop:d5-family} gives the five-dimensional simplex $\Delta_{3,118}$ has
\[\A_2(\Delta_{3,118};1)=-\frac1{24}.\]

At $q=19$, the coefficient of $t^3$ in $L_{\Delta_{3,q}}(t)$ is $-1/24$, but all coefficients in $L_{\Delta_{3,q}}(1+x)$ are positive. Thus the shift from $0$ to $1$ completely repairs the coefficient sign pattern, making the polynomial Taylor positive at the new center.

At $q=118$, the coefficient of $t^3$ at $0$ equals $\frac{18-118}{24}=-\frac{25}{6}$,
while the coefficient of $x^3$ at $1$ equals $(118+30)/24>0$.  However, the quadratic coefficient at $1$ is
$-1/24$.  Thus the original cubic negativity disappears and quadratic negativity emerges.  This illustrates the downward propagation of negativity.

\section{Optimal centers in every dimension}\label{sec:optimal-universal-shift}

For odd $d$, the estimate in \cref{prop:dimension-bound} is sharp.  For $d=2m$, it gives the center $m$, whereas the optimal value is $m-1$. Consequently, the optimal universal center is the same in two consecutive dimensions.
The principal result of this section is
\begin{equation*}
 \mathfrak u_d=\mathfrak u_d^+ =\left\lfloor\frac{d-1}{2}\right\rfloor.
\end{equation*}
The existence of these minima for $\mathfrak u_d$ and $\mathfrak u_d^+$ follows from \cref{prop:dimension-bound}, and clearly $\mathfrak u_d\leq\mathfrak u_d^+$.

\subsection{A coefficientwise inequality for shifted rising factorials}

For a polynomial $F(x)$, write $[x^\ell]F(x)$ for the coefficient of
$x^\ell$ in $F(x)$.  Recall the rising factorial
\[
 x^{\overline d}:=x(x+1)\cdots(x+d-1).
\]

\begin{lemma}\label{lem:shifted-rising-comparison}
Let $d=2m\geq2$, let $1\leq\ell\leq d-1$, and set
\[C_{d,\ell}:=[x^\ell]x^{\overline d}.
\]
Then $C_{d,\ell}>0$, and the following two inequalities hold.
\begin{enumerate}
\item[(a)] For every integer $a$ with $0\leq a\leq m$, we have 
\begin{equation}\label{eq:positive-rising-shift}
[x^\ell](x+a)^{\overline d}\geq\frac{2a+1}{d-1}\,C_{d,\ell}.
\end{equation}
\item[(b)] For every integer $b$ with $1\leq b\leq m$, we have 
\begin{equation}\label{eq:negative-rising-shift}
[x^\ell](x-b)^{\overline d}\geq-\frac{2b-1}{d-1}\,C_{d,\ell}.
\end{equation}
\end{enumerate}
\end{lemma}
\begin{proof}
Because $x^{\overline d}=x\prod_{\nu=1}^{d-1}(x+\nu)$, every coefficient of positive degree is positive.  Hence
$C_{d,\ell}>0$ for $1\leq\ell\leq d$.
We first prove part~(a).  Set $q=d-\ell$.  Since $1\leq\ell\leq d-1$, we have $1\leq q\leq d-1$.  Define
$E_q(a):=e_q(a,a+1,\ldots,a+d-1)$.
As usual, $e_0=1$. Expanding the product $(x+a)^{\overline d}$ shows that
\begin{equation}\label{eq:Eq-coefficient}
 E_q(a)=[x^\ell](x+a)^{\overline d}.
\end{equation}
Moreover, $E_q(0)=e_q(0,1,\ldots,d-1)=e_q(1,\ldots,d-1)=C_{d,\ell}$.
Every summand in
\[E_q(a)=\sum_{\substack{S\subseteq\{0,\ldots,d-1\}\\|S|=q}}\prod_{\nu\in S}(a+\nu)\]
is a polynomial in $a$ with nonnegative coefficients.  Consequently $E_q(a)$ itself has nonnegative coefficients, and hence
\begin{equation}\label{eq:Eq-linear-lower-bound}
 E_q(a)\geq E_q(0)+aE_q'(0)\qquad(a\geq0).
\end{equation}

We next estimate $E_q'(0)$.  Differentiating an elementary symmetric
function after shifting all its variables gives
\[E_q'(a)=\sum_{\nu=0}^{d-1}e_{q-1}(a,a+1,\ldots,\widehat{a+\nu},\ldots,a+d-1),
\]
where the hat denotes omission.  A fixed monomial of degree $q-1$
occurs in exactly $d-(q-1)=d-q+1$ of the summands.  Therefore, we have 
\begin{equation}\label{eq:Eq-derivative}
 E_q'(a)=(d-q+1)e_{q-1}(a,a+1,\ldots,a+d-1).
\end{equation}
At $a=0$, this becomes $E_q'(0)=(d-q+1)e_{q-1}(1,\ldots,d-1)$.

We claim that
\begin{equation}\label{eq:elementary-ratio-bound}
 e_q(1,\ldots,d-1)
 \leq\frac{(d-q+1)(d-1)}2e_{q-1}(1,\ldots,d-1).
\end{equation}
To prove the claim, set $y_\nu=\nu$ for $1\leq\nu\leq d-1$.
Double counting pairs consisting of a $q$-element set and one of its
elements gives
\begin{align*}
 q e_q(y_1,\ldots,y_{d-1})
=\sum_{\substack{T\subseteq\{1,\ldots,d-1\}\\|T|=q-1}}\left(\prod_{\nu\in T}y_\nu\right)\left(\sum_{\mu\notin T}y_\mu\right)
&\leq\left(\sum_{\mu=1}^{d-1}\mu\right)e_{q-1}(y_1,\ldots,y_{d-1})
\\& =\frac{d(d-1)}2e_{q-1}(y_1,\ldots,y_{d-1}).
\end{align*}
Thus
\[e_q(1,\ldots,d-1)\leq\frac{d(d-1)}{2q}e_{q-1}(1,\ldots,d-1).
\]
Finally, we have $q(d-q+1)-d=(q-1)(d-q)\geq0$, so $d/q\leq d-q+1$.  This proves \eqref{eq:elementary-ratio-bound}.
Combining \eqref{eq:Eq-derivative} and \eqref{eq:elementary-ratio-bound}, we obtain
$E_q'(0)\geq\frac2{d-1}E_q(0)$.
Substitution into \eqref{eq:Eq-linear-lower-bound} yields
\begin{align*}
E_q(a)\geq E_q(0)\left(1+\frac{2a}{d-1}\right)=\frac{d-1+2a}{d-1}E_q(0)\geq\frac{2a+1}{d-1}E_q(0),
\end{align*}
where the last inequality uses $d\geq2$.  Recalling \eqref{eq:Eq-coefficient} and $E_q(0)=C_{d,\ell}$ proves \eqref{eq:positive-rising-shift}.

We now prove part~(b).  Fix $1\leq b\leq m=d/2$, and set $\gamma_b:=\frac{2b-1}{d-1}$.
The two rising factorials have the factorizations
\begin{align*}
 (x-b)^{\overline d}=x^{\overline{d-b}}\prod_{\nu=1}^{b}(x-\nu),\quad
 x^{\overline d}=x^{\overline{d-b}}\prod_{\nu=d-b}^{d-1}(x+\nu).
\end{align*}
Therefore, we have 
\begin{equation}\label{eq:tail-factorization}
 (x-b)^{\overline d}+\gamma_bx^{\overline d}=x^{\overline{d-b}}H_{d,b}(x),
\end{equation}
where
\[ H_{d,b}(x) :=\prod_{\nu=1}^{b}(x-\nu) +\gamma_b\prod_{\nu=d-b}^{d-1}(x+\nu).
\]
We show that every coefficient of $H_{d,b}(x)$ is nonnegative.  For $0\leq j\leq b$, the coefficient of $x^{b-j}$ is
\begin{equation}\label{eq:H-coefficient}
 (-1)^je_j(1,\ldots,b)+\gamma_be_j(d-b,\ldots,d-1).
\end{equation}
If $j$ is even, this is plainly nonnegative.  Suppose that $j$ is odd.
For $1\leq\nu\leq b$, one has
\[ b(d-b+\nu-1)-(d-1)\nu =(d-b-1)(b-\nu)\geq0.
\]
Thus
\begin{equation}\label{eq:entrywise-high-low}
d-b+\nu-1\geq\frac{d-1}{b}\nu\qquad(1\leq\nu\leq b).
\end{equation}
Multiplying \eqref{eq:entrywise-high-low} over each $j$-element subset and then summing gives
\begin{equation}\label{eq:elementary-high-low}
 e_j(d-b,\ldots,d-1)\geq\left(\frac{d-1}{b}\right)^je_j(1,\ldots,b).
\end{equation}
Since $b\leq d/2$, we have $(d-1)/b\geq1$.  Since $j\geq1$,
\[\left(\frac{d-1}{b}\right)^j\geq\frac{d-1}{b}\geq\frac{d-1}{2b-1}=\frac1{\gamma_b}.
\]
It follows from \eqref{eq:elementary-high-low} that
\[\gamma_be_j(d-b,\ldots,d-1)\geq e_j(1,\ldots,b).
\]
Hence \eqref{eq:H-coefficient} is nonnegative also when $j$ is odd.
We have proved that $H_{d,b}(x)$ has nonnegative coefficients.  The polynomial $x^{\overline{d-b}}$ also has nonnegative coefficients. Equation \eqref{eq:tail-factorization} therefore shows, coefficient by
coefficient, that
\[[x^\ell](x-b)^{\overline d}+\gamma_b[x^\ell]x^{\overline d}\geq0.
\]
Since $[x^\ell]x^{\overline d}=C_{d,\ell}$, this is precisely \eqref{eq:negative-rising-shift}.
\end{proof}

\begin{lemma}\label{lem:even-comparison}
Let $d=2m\geq2$, let $B=m-1$, and let $g_r(B,\ell)$ be defined as in
\eqref{eq:g-definition}.  For every $1\leq\ell\leq d-1$, set
$\lambda_{d,\ell}:=\frac{[x^\ell]x^{\overline d}}{d-1}>0$.
Then we have $g_r(B,\ell)\geq\lambda_{d,\ell}(d+1-2r)$ for every $0\leq r\leq d$.
\end{lemma}
\begin{proof}
By the definition of $g_r$, we have 
\begin{align*}
g_r(B,\ell)=[x^\ell]\prod_{\nu=1}^{d}(x+B-r+\nu)=[x^\ell]\prod_{\nu=1}^{d}(x+m-1-r+\nu)=[x^\ell](x+m-r)^{\overline d}.
\end{align*}
First suppose that $0\leq r\leq m$, and set $a=m-r$.  Then
$0\leq a\leq m$ and $d+1-2r=2m+1-2r=2a+1$.
Part~(a) of \cref{lem:shifted-rising-comparison} gives
\[g_r(B,\ell)\geq\frac{2a+1}{d-1}[x^\ell]x^{\overline d}=\lambda_{d,\ell}(d+1-2r).
\]
Now suppose that $m+1\leq r\leq d$, and set $b=r-m$.  Then
$1\leq b\leq m$ and $d+1-2r=1-2b=-(2b-1)$.
Part~(b) of \cref{lem:shifted-rising-comparison} gives
\[g_r(B,\ell)\geq-\frac{2b-1}{d-1}[x^\ell]x^{\overline d}=\lambda_{d,\ell}(d+1-2r).
\]
This covers every $0\leq r\leq d$.
\end{proof}

\subsection{Taylor positivity in even dimension}

\begin{theorem}\label{thm:even-dimensional-optimal-upper}
Let $P$ be a $2m$-dimensional lattice polytope, where $m\geq1$.  Then $P$ is Taylor positive at $m-1$:
\[
 \A_\ell(P;m-1)>0\qquad(0\leq\ell\leq2m).
\]
Equivalently, $L_P(m-1+x)$ has positive coefficients in every degree.
\end{theorem}
\begin{proof}
Set $d=2m$.
For $\ell=0$, the center $m-1$ is a nonnegative integer.  Hence
\[\A_0(P;m-1)=L_P(m-1)=|(m-1)P\cap\Z^d|>0.
\]
This remains true for $m-1=0$, because $0P=\{0\}$.
For $1\leq\ell\leq d-1$, equation \eqref{eq:derivative-g} gives
\begin{equation}\label{eq:even-derivative-sum}
 L_P^{(\ell)}(m-1)=\frac{\ell!}{d!}\sum_{r=0}^{d}h_r^*g_r(m-1,\ell).
\end{equation}
Apply \cref{lem:even-comparison}.  Since every $h_r^*$ is nonnegative, we have 
\[\sum_{r=0}^{d}h_r^*g_r(m-1,\ell)\geq\lambda_{d,\ell}\sum_{r=0}^{d}h_r^*(d+1-2r).
\]
By \eqref{eq:surface-hstar}, we have
\[\sum_{r=0}^{d}h_r^*(d+1-2r)=2(d-1)!c_{d-1}>0.
\]
Because $\lambda_{d,\ell}>0$, we conclude that the sum on the left is positive.  Equation \eqref{eq:even-derivative-sum} now implies
$L_P^{(\ell)}(m-1)>0$ for $1\leq\ell\leq d-1$. Finally, we obtain $L_P^{(d)}(m-1)=d!c_d>0$ because $c_d>0$. 
Therefore, we have 
\[\A_\ell(P;m-1)=\frac{L_P^{(\ell)}(m-1)}{\ell!}>0\qquad(0\leq\ell\leq d),
\]
as required.
\end{proof}

\begin{remark}
For $d=4$, the conclusion at center $1$ is already contained in the
coefficient estimates in \cite[Proposition~4.7]{BeckEtAl2005}.  In arbitrary dimension, Beck,
De Loera, Develin, Pfeifle, and Stanley proved positivity at the center $\lfloor d/2\rfloor$; see
\cite[Theorem~1.2(b) and its proof]{BeckEtAl2005}.  To the best of our knowledge, the improvement to $m-1$ for every even dimension $d=2m$, and hence the exact universal shift below, have not previously been recorded.  For recent accounts that continue to state
$\lfloor d/2\rfloor$ as the general real-root endpoint, see \cite{FerroniHigashitani2024,Koelbl2025}.
\end{remark}

The theorem improves the dimension part of \cref{cor:combined-bound}.

\begin{theorem}\label{cor:improved-combined-bound}
Let $P$ be a $d$-dimensional lattice polytope, and let
$s=\deg h_P^*(z)$.  Then
\[ 0\leq\tau(P)\leq\tau^+(P)\leq\min\left\{\max\{0,s-1\},\left\lfloor\frac{d-1}{2}\right\rfloor\right\}.
\]
\end{theorem}
\begin{proof}
The degree estimate is \cref{thm:degree-bound}.  If $d$ is odd, then $\lfloor\frac{d-1}{2}\rfloor=\lfloor\frac d2\rfloor$,
so the dimension estimate follows from \cref{prop:dimension-bound}.  If $d=2m$ is even, it follows from \cref{thm:even-dimensional-optimal-upper}.
\end{proof}

\begin{theorem}\label{thm:all-dimensional-optimal-shift}
For every $d\geq1$, we have 
\[\mathfrak u_d=\mathfrak u_d^+=\left\lfloor\frac{d-1}{2}\right\rfloor.
\]
Equivalently, if $d=2m-1$ or $d=2m$, then $\mathfrak u_d=\mathfrak u_d^+=m-1$.
\end{theorem}
\begin{proof}
Set $K_d=\lfloor(d-1)/2\rfloor$.  By \cref{cor:improved-combined-bound}, every $d$-dimensional lattice
polytope is Taylor positive at $K_d$.  Therefore
\begin{equation}\label{eq:universal-upper-K}
 \mathfrak u_d\leq\mathfrak u_d^+\leq K_d.
\end{equation}
For $d=1$ and $d=2$, we have $K_d=0$.  Since the universal centers are nonnegative integers, equality follows immediately from
\eqref{eq:universal-upper-K}.
Suppose now that $d\geq3$, and set $s=\lceil d/2\rceil$.  Then $s\geq2$, $2s-1\leq d$, and
$s-1=\lfloor\frac{d-1}{2}\rfloor=K_d$.
Choose $q$ satisfying the large-$q$ condition in \cref{thm:pyramid-sharpness}, and consider
$\Delta_{s,q}^{(d)}=\operatorname{pyr}^{\,d-(2s-1)}(\Delta_{s,q})$.
By \cref{thm:pyramid-sharpness}, we have 
\[\tau(\Delta_{s,q}^{(d)})=\tau^+(\Delta_{s,q}^{(d)})=s-1=K_d.\]
Consequently, no integer smaller than $K_d$ works universally, either for nonnegativity or for positivity.  Thus
$K_d\leq\mathfrak u_d\leq\mathfrak u_d^+\leq K_d$, which proves the result.
\end{proof}

\begin{corollary}\label{cor:improved-real-root-upper-bound}
Let $P$ be a $d$-dimensional lattice polytope.  Then every real root
$\alpha$ of $L_P(t)$ satisfies $\alpha<\lfloor\frac{d-1}{2}\rfloor$. Together with the lower bound in
\cite[Theorem~1.2(b)]{BeckEtAl2005}, this shows that every real Ehrhart root lies in
\[\left[-d,\left\lfloor\frac{d-1}{2}\right\rfloor\right).
\]
\end{corollary}
\begin{proof}
Set $K=\lfloor(d-1)/2\rfloor$.  By \cref{cor:improved-combined-bound},
$L_P(K+x)=\sum_{j=0}^{d}\A_j(P;K)x^j$ has positive coefficients.  Hence $L_P(K+x)>0$ for every $x\geq0$, and $L_P$ has no real root in $[K,\infty)$.  The lower bound $-d$ is the lower half of \cite[Theorem~1.2(b)]{BeckEtAl2005}.
\end{proof}

\begin{remark}
For $d=2m\geq4$, the sharp examples may be taken to be the single lattice pyramids $\Delta_{m,q}^{(2m)}=\operatorname{pyr}(\Delta_{m,q})$.
They satisfy
$h_{\Delta_{m,q}^{(2m)}}^*(z)=1+(q-1)z^m$
and
\[L_{\Delta_{m,q}^{(2m)}}(t)=\binom{t+2m}{2m}+(q-1)\binom{t+m}{2m}.
\]
For all sufficiently large $q$, they satisfy $\tau(\Delta_{m,q}^{(2m)})=\tau^+(\Delta_{m,q}^{(2m)})=m-1$.
\end{remark}

\section{Real and integer Taylor-nonnegativity indices}\label{sec:root-geometry}

The invariant $\rho(P)$ determines $\tau(P)$ and $\tau^+(P)$, except that positivity at the origin must be recorded separately.

\begin{proposition}\label{prop:natural-rounding}
Let $P$ be a lattice polytope and set $r=\rho(P)$.  Then $\tau(P)=\lceil r\rceil$.
Moreover,
\[\tau^+(P)=
 \begin{cases}
  0,&r=0\text{ and }\A_j(P;0)>0\text{ for every }j,\\
  1,&r=0\text{ and }\A_j(P;0)=0\text{ for at least one }j,\\
  \lceil r\rceil,&r\notin\Z,\\
  r+1,&r\in\Z_{>0}.
 \end{cases}
\]
Equivalently, if $r>0$, then the set of real centers at which all Taylor coefficients are positive is $(r,\infty)$.
\end{proposition}
\begin{proof}
By \eqref{eq:real-positivity-ray}, simultaneous Taylor nonnegativity holds
at precisely the nonnegative real centers in $[r,\infty)$.  The least
nonnegative integer in this interval is $\lceil r\rceil$, proving the
formula for $\tau(P)$.
At the center $r$, all coefficients are nonnegative.  If $\delta>0$, the
translation formula gives
\[\A_j(P;r+\delta)=\sum_{\ell=j}^{d}\binom{\ell}{j}\delta^{\ell-j}\A_\ell(P;r).
\]
Every summand is nonnegative, while the summand with $\ell=d$ is
\[\binom{d}{j}\delta^{d-j}\A_d(P;r)=\binom{d}{j}\delta^{d-j}c_d>0.
\]
Thus every coefficient is positive at every center greater than $r$.

Suppose that $r>0$ and that all coefficients were positive at $r$.  Since the finitely many coefficient functions
$k\mapsto\A_j(P;k)$ are continuous, they would remain positive on an interval $(r-\varepsilon,r+\varepsilon)$ for some $\varepsilon>0$. After replacing $\varepsilon$ by $\min\{\varepsilon,r/2\}$, this would give simultaneous nonnegativity at the nonnegative center $r-\varepsilon$, contradicting the definition of $r$.  Consequently, if $r>0$, at least one coefficient vanishes at $r$.  The asserted cases for $\tau^+(P)$ now follow by taking the least integer greater than $r$, with the separate endpoint possibility at $r=0$.
\end{proof}

\subsection{The derivative-root envelope}

For a nonconstant polynomial $p\in\R[t]$ with positive leading coefficient,
define
\[\lambda_{\R}(p):=\max\{x\in\R:p(x)=0\}.
\]
If $p$ has no real zero, we set $\lambda_{\R}(p)=-\infty$.  We use the same
convention for a positive constant polynomial.

\begin{definition}\label{def:derivative-root-envelope}
Let $P$ be a $d$-dimensional lattice polytope, where $d\geq1$.  Its
\emph{derivative-root envelope} is
\begin{equation}\label{eq:derivative-root-envelope}
 \mathcal R(P):=\max_{0\leq j<d}\lambda_{\R}\!\left(\A_j(P;\,\cdot\,)\right)=\max_{0\leq j<d}\lambda_{\R}\!\left(L_P^{(j)}\right).
\end{equation}
Multiplication by the positive constant $1/j!$ does not change the zeros, which explains the second equality.
\end{definition}

Every polynomial $k\mapsto\A_j(P;k)$ has degree $d-j$ and positive leading coefficient $\binom{d}{j}c_d$.
In particular, $\A_{d-1}(P;k)$ is a nonconstant linear polynomial, so $\mathcal R(P)$ is a finite real number.

\begin{theorem}\label{thm:root-envelope}
For every lattice polytope $P$ of positive dimension, we have 
\begin{equation}\label{eq:rho-root-envelope}
\rho(P)=\max\{0,\mathcal R(P)\}.
\end{equation}
More precisely, on the whole real line one has
\begin{equation}\label{eq:absolute-monotonicity-ray-whole-line}
\left\{k\in\R:\A_j(P;k)\geq0\text{ for every }0\leq j\leq d\right\}=[\mathcal R(P),\infty).
\end{equation}
At $k=\mathcal R(P)$ at least one coefficient vanishes, and all coefficients are positive for $k>\mathcal R(P)$.
\end{theorem}
\begin{proof}
Fix an index $j<d$.  Because $\A_j(P;k)$ has positive leading coefficient, it is positive for all
sufficiently large $k$.  It has no real zero larger than $\lambda_{\R}(\A_j(P;\,\cdot\,))$, and hence
\[\A_j(P;k)>0\qquad\text{for every }k>\lambda_{\R}(\A_j(P;\,\cdot\,)).
\]
Indeed, its sign cannot change on the zero-free interval to the right of its largest real zero, and that sign is positive near infinity.  If
$\lambda_{\R}(\A_j(P;\,\cdot\,))=\mathcal R(P)$, then $\A_j(P;\mathcal R(P))=0$; if
$\lambda_{\R}(\A_j(P;\,\cdot\,))<\mathcal R(P)$, then $\A_j(P;\mathcal R(P))>0$.  Consequently,
\[\A_j(P;\mathcal R(P))\geq0\qquad(0\leq j<d).\]
The final coefficient is the positive constant $\A_d(P;k)=c_d$.  Thus all shifted coefficients are nonnegative at $\mathcal R(P)$.  The translation formula \eqref{eq:translation-semigroup} then shows that they are positive at every $k>\mathcal R(P)$.

It remains to prove that no $k<\mathcal R(P)$ can have all shifted coefficients nonnegative.  Suppose, to the contrary, that
\[\A_\ell(P;k)\geq0\qquad(0\leq\ell\leq d)
\]
for some $k<\mathcal R(P)$.  By the definition of $\mathcal R(P)$, there is an index $j<d$ for which
\[y:=\lambda_{\R}(\A_j(P;\,\cdot\,))>k.
\]
Apply the translation formula with displacement $y-k>0$:
\[\A_j(P;y)=\sum_{\ell=j}^{d}\binom{\ell}{j}(y-k)^{\ell-j}\A_\ell(P;k).
\]
Every summand is nonnegative.  Moreover, the summand with $\ell=d$ is positive:
\[\binom{d}{j}(y-k)^{d-j}\A_d(P;k)=\binom{d}{j}(y-k)^{d-j}c_d>0.
\]
It follows that $\A_j(P;y)>0$.  This contradicts the definition of $y$, according to which $\A_j(P;y)=0$.  Therefore the set in
\eqref{eq:absolute-monotonicity-ray-whole-line} is exactly $[\mathcal R(P),\infty)$.

Intersecting this ray with $\R_{\geq0}$ gives $\mathcal N(P)=[\max\{0,\mathcal R(P)\},\infty)$, and comparison with \eqref{eq:real-positivity-ray} proves \eqref{eq:rho-root-envelope}.  Finally, the maximum in
\eqref{eq:derivative-root-envelope} is attained by at least one index $j<d$, so at least one coefficient vanishes at $\mathcal R(P)$.
\end{proof}

The theorem also determines $\tau(P)$ and $\tau^+(P)$ exactly.

\begin{corollary}\label{cor:integer-center-root-formula}
We have 
\begin{equation}\label{eq:tau-root-formula}
 \tau(P)=\left\lceil\max\{0,\mathcal R(P)\}\right\rceil
\end{equation}
and
\begin{equation*}
 \tau^+(P)=
 \begin{cases}
  0,&\mathcal R(P)<0,\\[2mm]
  \lfloor \mathcal R(P)\rfloor+1,&\mathcal R(P)\geq0.
 \end{cases}
\end{equation*}
Consequently, $\tau(P)\neq\tau^+(P)$ precisely when either $\mathcal R(P)=0$ or $\mathcal R(P)$ is a positive integer.
\end{corollary}
\begin{proof}
By \cref{thm:root-envelope}, Taylor nonnegativity holds at a real center $k\geq0$ exactly when
\[ k\geq\max\{0,\mathcal R(P)\}.
\]
The least integer satisfying this inequality is the ceiling in \eqref{eq:tau-root-formula}.

When $\mathcal R(P)<0$, every coefficient at $0$ is positive by \cref{thm:root-envelope}; hence $\tau^+(P)=0$.  Suppose instead that $\mathcal R(P)\geq0$.  At $\mathcal R(P)$ at least one coefficient vanishes, whereas every
coefficient is positive at every real center larger than $\mathcal R(P)$. Thus the least integral center of positivity is the least integer larger than $\mathcal R(P)$, namely $\lfloor \mathcal R(P)\rfloor+1$.
\end{proof}

\subsection{Comparison with the roots of the Ehrhart polynomial}

The real-rootedness of the Ehrhart polynomial and the real-rootedness of the $h^*$-polynomial are distinct properties.  Recent examples illustrating the subtle relationships among Ehrhart positivity, Ehrhart roots, and $h^*$-real-rootedness are surveyed by Ferroni and Higashitani \cite{FerroniHigashitani2024}.

For a nonconstant polynomial $p\in\mathbb C[t]$, define its spectral abscissa by
\[\beta(p):=\max\{\Re(\zeta):p(\zeta)=0\},
\]
where $\Re(\zeta)$ denotes the real part of $\zeta$. The classical Gauss--Lucas theorem says that the zeros of $p'$ lie in the
convex hull of the zeros of $p$; see, for example, \cite{RahmanSchmeisser2002}.  We recall a short proof in the form needed here.

\begin{lemma}[Gauss--Lucas]\label{lem:gauss-lucas}
Let $p(t)=a\prod_{\nu=1}^{n}(t-\alpha_\nu)$, $a\neq0$, where the roots are repeated according to multiplicity.  Every zero of $p'(t)$ belongs to the convex hull of $\alpha_1,\ldots,\alpha_n$.  Consequently,
\[\beta\!\left(p^{(j)}\right)\leq\beta(p)\qquad(0\leq j<n).
\]
\end{lemma}
\begin{proof}
Let $\zeta$ be a zero of $p'(t)$.  If $\zeta$ is also a zero of $p(t)$, then it is one of the $\alpha_\nu$ and there is nothing to prove.  Otherwise, logarithmic differentiation gives
\[0=\frac{p'(\zeta)}{p(\zeta)} =\sum_{\nu=1}^{n}\frac1{\zeta-\alpha_\nu}.
\]
Since
\[\frac1{\zeta-\alpha_\nu}=\frac{\overline{\zeta}-\overline{\alpha_\nu}}{|\zeta-\alpha_\nu|^2},
\]
we obtain
\[
 \overline{\zeta}
 \sum_{\nu=1}^{n}\frac1{|\zeta-\alpha_\nu|^2}
 =
 \sum_{\nu=1}^{n}
 \frac{\overline{\alpha_\nu}}{|\zeta-\alpha_\nu|^2}.
\]
Taking complex conjugates yields
\[
 \zeta
 =
 \frac{\displaystyle
       \sum_{\nu=1}^{n}
       \frac{\alpha_\nu}{|\zeta-\alpha_\nu|^2}}
      {\displaystyle
       \sum_{\nu=1}^{n}
       \frac1{|\zeta-\alpha_\nu|^2}}.
\]
This is a convex combination of the roots $\alpha_\nu$, because all the weights are positive.  Iterating the assertion gives the result for every higher derivative.
\end{proof}

\begin{theorem}\label{thm:rho-root-sandwich}
We have 
\begin{equation}\label{eq:rho-root-sandwich}
\max\{0,\lambda_{\R}(L_P)\}\leq \rho(P)\leq \max\{0,\beta(L_P)\}.
\end{equation}
If $L_P(t)$ is real-rooted, then both bounds coincide and
\begin{equation}\label{eq:rho-real-rooted}
 \rho(P)=\max\{0,\lambda_{\R}(L_P)\}.
\end{equation}
\end{theorem}

\begin{proof}
The index $j=0$ occurs in the maximum defining $\mathcal R(P)$, so $\lambda_{\R}(L_P)\leq\mathcal R(P)$.
This gives the first inequality after applying \cref{thm:root-envelope}.
By \cref{lem:gauss-lucas}, every root of every derivative $L_P^{(j)}(t)$ lies in the convex hull of the roots of $L_P(t)$.  In particular, every real root $x$ of $L_P^{(j)}(t)$ satisfies $x=\Re(x)\leq\beta(L_P)$.
It follows that $\mathcal R(P)\leq\beta(L_P)$, proving the second inequality.
Suppose now that $L_P(t)$ is real-rooted.  Rolle's theorem implies that the largest zero of $L_P^{(j+1)}(t)$ is at most the largest zero of $L_P^{(j)}(t)$.  Hence
\[\lambda_{\R}(L_P)\geq\lambda_{\R}(L_P')\geq\cdots\geq\lambda_{\R}(L_P^{(d-1)}).
\]
Therefore $\mathcal R(P)=\lambda_{\R}(L_P)$, and \eqref{eq:rho-real-rooted} follows from \cref{thm:root-envelope}.
\end{proof}

\begin{corollary}\label{cor:left-half-plane}
If every zero of $L_P(t)$ has negative real part, then $\rho(P)=\tau(P)=\tau^+(P)=0$.
Equivalently, $P$ is Taylor positive at every real center $k\geq0$.
\end{corollary}
\begin{proof}
The hypothesis says $\beta(L_P)<0$. 
By Gauss–Lucas lemma, $\mathcal{R}(P) \leq \beta(L_P) < 0$. Hence \cref{thm:rho-root-sandwich} gives $\rho(P) = 0$, and \cref{cor:integer-center-root-formula} yields $\tau(P) = \tau^+(P) = 0$. Moreover, \cref{thm:rho-root-sandwich} shows that all Taylor coefficients are positive for every $k \geq 0$.
\end{proof}

For real-rooted $L_P(t)$, Ehrhart positivity has a simple root-theoretic interpretation.

\begin{corollary}\label{cor:real-rooted-equivalences}
Suppose $L_P(t)$ is real-rooted.  The following conditions are equivalent:
\begin{enumerate}
 \item[(i)] $P$ is Ehrhart positive;
 \item[(ii)] $\rho(P)=0$;
 \item[(iii)] $\tau^+(P)=0$;
 \item[(iv)] every zero of $L_P(t)$ is negative.
\end{enumerate}
\end{corollary}
\begin{proof}
If all zeros are negative, then $L_P(t)=c_d\prod_{\nu=1}^{d}(t+a_\nu)$ with $a_\nu>0$, so every monomial coefficient is positive.  Thus (iv) implies (i), and (i) plainly implies (iii).  Condition (iii) implies (ii).
Finally, if $\rho(P)=0$, then \eqref{eq:rho-real-rooted} gives $\lambda_{\R}(L_P)\leq0$.  Since $L_P(0)=1$, zero is not a root.
Therefore every root is negative, proving that (ii) implies (iv).
\end{proof}

\section{Properties of Taylor coefficients}\label{Section-Prop-Taylor}

This section mainly explores various properties of the Taylor coefficients $\A_{j}(P;k)$ as polynomials in $k$.

\subsection{Interlacing property}

We first state exactly what remains true without real-rootedness.

\begin{proposition}\label{prop:unconditional-rolle}
Fix $0\leq j<d$, and let $x_1<x_2<\cdots<x_r$ be the distinct real zeros of $k\mapsto\A_j(P;k)$, with respective multiplicities $m_1,\ldots,m_r$.  Then:
\begin{enumerate}
 \item[(i)] if $m_i\geq2$, then $x_i$ is a zero of $\A_{j+1}(P;\,\cdot\,)$ of multiplicity $m_i-1$;
 \item[(ii)] for every $1\leq i<r$, the polynomial $\A_{j+1}(P;\,\cdot\,)$ has at least one zero in $(x_i,x_{i+1})$.
\end{enumerate}
\end{proposition}

\begin{proof}
Set $p(k)=\A_j(P;k)$.  The differential equation $p'(k)=(j+1)\A_{j+1}(P;k)$ shows that $p'(k)$ and $\A_{j+1}(P;k)$ have the same zeros with the same multiplicities. 
If $x_i$ is a zero of $p(k)$ of multiplicity $m_i$, write
\[p(k)=(k-x_i)^{m_i}q(k),\qquad q(x_i)\neq0.
\]
Differentiation gives
\[ p'(k)=(k-x_i)^{m_i-1}\bigl(m_iq(k)+(k-x_i)q'(k)\bigr).
\]
The factor in parentheses has value $m_iq(x_i)\neq0$ at $x_i$, so $p'(k)$ has multiplicity exactly $m_i-1$ there.  This proves (i).
Statement (ii) is Rolle's theorem applied on each interval $[x_i,x_{i+1}]$.
\end{proof}

\begin{theorem}\label{thm:derivative-interlacing}
Fix $0\leq j< d-1$.
Suppose $\A_j(P;\,\cdot\,)$ has only real zeros.  Then $\A_{j+1}(P;\,\cdot\,)$ has only real zeros and its zeros weakly
interlace those of $\A_j(P;\,\cdot\,)$.  If all zeros of $\A_j(P;\,\cdot\,)$ are simple, the interlacing is strict.
In particular, if $L_P$ is real-rooted and $\alpha_{j,1}\leq\alpha_{j,2}\leq\cdots\leq\alpha_{j,d-j}$ 
are the zeros of $\A_j(P;\,\cdot\,)$, repeated according to multiplicity, then
\[\alpha_{j,1}\leq\alpha_{j+1,1}\leq\alpha_{j,2}\leq\cdots\leq\alpha_{j+1,d-j-1}\leq\alpha_{j,d-j}.
\]
\end{theorem}
\begin{proof}
Write the distinct zeros of $p(k)=\A_j(P;k)$ as $x_1<\cdots<x_r$, with multiplicities $m_1,\ldots,m_r$.  Since $p(k)$ is real-rooted, $m_1+\cdots+m_r=d-j$.
By \cref{prop:unconditional-rolle}, the derivative has roots of total multiplicity
\[\sum_{i=1}^{r}(m_i-1)=d-j-r
\]
at the points $x_i$, and at least one further root in each of the $r-1$ open intervals $(x_i,x_{i+1})$.  Thus we have accounted for at least $(d-j-r)+(r-1)=d-j-1$ real zeros of $p'(k)$, counted with multiplicity.  Since $\deg p'(k)=d-j-1$, these are all of its zeros.  Hence $p'(k)$ is real-rooted, with exactly the stated weak interlacing.  If the roots of $p(k)$ are simple, then $r=d-j$, no
root is inherited at an endpoint, and the $d-j-1$ derivative zeros lie strictly inside the $d-j-1$ gaps.
This completes the proof.
\end{proof}

It is sometimes useful to isolate the point at which the derivative chain
becomes real-rooted.
The \emph{hyperbolicity index} of $P$ is
\[\eta(P):=\min\{j\in\{0,\ldots,d-1\}:L_P^{(j)}(t)\text{ is real-rooted}\}.
\]
This index is well defined because $L_P^{(d-1)}(t)$ is linear.

\begin{corollary}\label{cor:interlacing-tail}
For every $\eta(P)\leq j\leq d-1$, the polynomial $L_P^{(j)}(t)$ is real-rooted,  and the zeros in the chain
\[L_P^{(\eta(P))}(t),L_P^{(\eta(P)+1)}(t),\ldots,L_P^{(d-1)}(t)
\]
interlace successively.  In particular,
\[\lambda_{\R}(L_P^{(\eta(P))})\geq\lambda_{\R}(L_P^{(\eta(P)+1)})\geq\cdots\geq\lambda_{\R}(L_P^{(d-1)}).
\]
\end{corollary}
\begin{proof}
Apply \cref{thm:derivative-interlacing} inductively, starting with $j=\eta(P)$.
\end{proof}

We now consider the five-dimensional simplex $\Delta_{3,19}$ from \cref{sec:sharp-family,sec:shift-one}.

\begin{example}\label{prop:q19-noninterlacing}
For the simplex $\Delta_{3,19}$, by \cref{prop:d5-family}, we have 
\[L_{\Delta_{3,19}}(t)=\frac{19}{120}t^5+\frac18t^4-\frac1{24}t^3+\frac{15}{8}t^2+\frac{173}{60}t+1.
\]
The polynomial $L_{\Delta_{3,19}}$ itself has no nonnegative real zero.
Through a tedious computation, one proves that $\A_2(\Delta_{3,19};k)=\frac{38k^3+18k^2-3k+45}{24}$ has exactly one real root.
It follows that $\rho(\Delta_{3,19})=\frac{\sqrt{74}-6}{38}>0$ and
\[\lambda_{\R}\!\left(\A_2(\Delta_{3,19};\,\cdot\,)\right)<0<\lambda_{\R}\!\left(\A_3(\Delta_{3,19};\,\cdot\,)\right).
\]
Thus the largest real zeros of successive coefficient functions need not form a decreasing sequence when the preceding coefficient polynomial is not real-rooted.
\end{example}

\subsection{Laguerre inequalities and Newton inequalities}

The classical Laguerre inequality states that if a real polynomial $p$ is real-rooted, then
\[p'(x)^2-p(x)p''(x)\geq0\qquad(x\in\R).
\]
We include its short proof because its translation into shifted Ehrhart coefficients is especially transparent.

\begin{lemma}[Laguerre inequality, \cite{Laguerre}]\label{lem:laguerre}
Let $p\in\R[t]$ be a nonconstant real polynomial with only real zeros.  Then
\[p'(x)^2-p(x)p''(x)\geq0 \qquad(x\in\R).
\]
If all zeros of $p$ are simple, the inequality is strict for every real $x$.
\end{lemma}
\begin{proof}
Write $p(x)=a\prod_{\nu=1}^{n}(x-\alpha_\nu)$, $\alpha_\nu\in\R$, with multiplicities included. 
Away from the zeros of $p(x)$, we have 
\[\frac{p'(x)}{p(x)}=\sum_{\nu=1}^{n}\frac1{x-\alpha_\nu}.
\]
Differentiating gives
\[\frac{p''(x)p(x)-(p'(x))^2}{p(x)^2}= \left(\frac{p'}p\right)'(x)=-\sum_{\nu=1}^{n}\frac1{(x-\alpha_\nu)^2}.
\]
Consequently,
\[ p'(x)^2-p(x)p''(x) = p(x)^2\sum_{\nu=1}^{n}\frac1{(x-\alpha_\nu)^2}>0
\]
whenever $p(x)\neq0$.  The left-hand side is a polynomial and hence is continuous, so the nonnegative inequality extends to the zeros of $p(x)$. At a simple zero $\alpha_\nu$, it has value $p'(\alpha_\nu)^2>0$. Thus it is everywhere strict when all roots are simple.
\end{proof}

\begin{theorem}\label{thm:laguerre-turan-shifted}
Fix $0\leq j\leq d-2$.  If the polynomial $k\mapsto\A_j(P;k)$ is real-rooted, then, for every $k\in\R$,
\begin{equation}\label{eq:laguerre-turan-A}
(j+1)\A_{j+1}(P;k)^2\geq(j+2)\A_j(P;k)\A_{j+2}(P;k).
\end{equation}
If $\A_j(P;\,\cdot\,)$ has only simple zeros, the inequality is strict.
\end{theorem}
\begin{proof}
Apply \cref{lem:laguerre} to $p(k)=\A_j(P;k)$.  The coefficient
differential equations give
\[ p'(k)=(j+1)\A_{j+1}(P;k)\]
and
\[ p''(k)=(j+1)(j+2)\A_{j+2}(P;k).
\]
Substitution into $p'(k)^2-p(k)p''(k)\geq0$ yields
\[(j+1)^2\A_{j+1}(P;k)^2-(j+1)(j+2)\A_j(P;k)\A_{j+2}(P;k)\geq0.
\]
Division by $j+1>0$ proves \eqref{eq:laguerre-turan-A}.
\end{proof}

This inequality can be written explicitly as a quadratic inequality in the $h^*$-vector.  Define
\begin{equation*}
\Phi_j(P;k):=\sum_{r=0}^{d}h_r^*e_{d-j}(k-r+1,k-r+2,\ldots,k-r+d).
\end{equation*}
By \eqref{eq:formula-elementary}, we have 
\begin{equation}\label{eq:Phi-A-relation}
 \Phi_j(P;k)=d!\A_j(P;k).
\end{equation}

\begin{corollary}\label{cor:laguerre-hstar}
Under the hypotheses of \cref{thm:laguerre-turan-shifted}, we have 
\begin{equation*}
(j+1)\Phi_{j+1}(P;k)^2\geq(j+2)\Phi_j(P;k)\Phi_{j+2}(P;k) \qquad(k\in\R).
\end{equation*}
Equivalently,
\begin{small}
\begin{align*}
&(j+1)\left(\sum_{r=0}^{d}h_r^* e_{d-j-1}(k-r+1,\ldots,k-r+d)\right)^2\\ &\quad\geq(j+2)\left(\sum_{r=0}^{d}h_r^*e_{d-j}(k-r+1,\ldots,k-r+d)\right)
 \times\left(\sum_{r=0}^{d}h_r^*e_{d-j-2}(k-r+1,\ldots,k-r+d)\right).
\end{align*}
\end{small}
\end{corollary}
\begin{proof}
Multiply \eqref{eq:laguerre-turan-A} by $(d!)^2$ and use \eqref{eq:Phi-A-relation}. This completes the proof.
\end{proof}

Inequalities among the coefficients of Ehrhart polynomials have also been extensively studied; see, e.g., \cite{BeckEtAl2005,Betke-McMullen,Hibi1995,Stanley1980,Stanley-Cohen-Macau}.
Some inequalities on $h^*$-polynomials are found in \cite[Chapter 10.4]{BeckRobins2015}.

A stronger coefficient inequality is available when the Ehrhart polynomial itself is real-rooted.  We use the classical Newton
inequalities: if $y_1,\ldots,y_d\geq0$ and $1\leq r\leq d-1$, then the normalized elementary symmetric functions satisfy
\[\left(\frac{e_r(y_1,\ldots,y_d)}{\binom dr}\right)^2\geq\frac{e_{r-1}(y_1,\ldots,y_d)}{\binom d{r-1}} \frac{e_{r+1}(y_1,\ldots,y_d)}{\binom d{r+1}}.
\]
See, for example, \cite{Branden2015,RahmanSchmeisser2002}.

\begin{theorem}\label{thm:newton-shifted}
Suppose $L_P(t)$ is real-rooted.  Then, for every $k\geq\rho(P)$ and every $1\leq j\leq d-1$, we have 
\begin{equation}\label{eq:newton-normalized-A}
\left(\frac{\A_j(P;k)}{\binom dj}\right)^2\geq\frac{\A_{j-1}(P;k)}{\binom d{j-1}}\frac{\A_{j+1}(P;k)}{\binom d{j+1}}.
\end{equation}
Equivalently,
\begin{equation}\label{eq:newton-factor-A}
\A_j(P;k)^2\geq\frac{(j+1)(d-j+1)}{j(d-j)}\A_{j-1}(P;k)\A_{j+1}(P;k).
\end{equation}
\end{theorem}
\begin{proof}
Let the real zeros of $L_P(t)$ be $\alpha_1,\ldots,\alpha_d$, repeated according to multiplicity.  By \eqref{eq:rho-real-rooted}, we have $k\geq\rho(P)\geq\max_\nu\alpha_\nu$.
Set $y_\nu=k-\alpha_\nu\geq0$.  Factoring at the center $k$ gives
\[L_P(k+x)=c_d\prod_{\nu=1}^{d}(x+y_\nu).
\]
The coefficient of $x^j$ in this product is $\A_j(P;k)=c_de_{d-j}(y_1,\ldots,y_d)$.
Apply Newton's inequality with $r=d-j$.  Since $\binom d{d-j}=\binom dj$, reversing the index gives precisely \eqref{eq:newton-normalized-A}. Finally,
\[\frac{\binom dj^2}{\binom d{j-1}\binom d{j+1}}=\frac{(j+1)(d-j+1)}{j(d-j)},
\]
which proves \eqref{eq:newton-factor-A}.
\end{proof}

We also obtain the following inequality concerning the coefficients of the $h^*$-polynomial.

\begin{corollary}\label{cor:newton-hstar}
Under the hypotheses of \cref{thm:newton-shifted}, for $k\geq\rho(P)$ and $1\leq j\leq d-1$, we have 
\begin{equation*}
\binom d{j-1}\binom d{j+1}\Phi_j(P;k)^2\geq\binom dj^2\Phi_{j-1}(P;k)\Phi_{j+1}(P;k).
\end{equation*}
\end{corollary}
\begin{proof}
Substitute $\Phi_m(P;k)=d!\A_m(P;k)$ into \eqref{eq:newton-normalized-A} and clear the positive denominators.
\end{proof}

\subsection{Products, dilations, and palindromicity}\label{Subsection-PDP-prob}

Let $P\subset \mathbb{R}^p$ and $Q\subset \mathbb{R}^q$ be lattice polytopes of dimensions $p$ and $q$. 
Then the \emph{Cartesian product} of $P$ and $Q$ is the convex polytope defined by
$$P\times Q\;=\;\left\{(x,y)\in\mathbb{R}^{p+q}\bigm| x\in P,\; y\in Q\right\}.$$
Then $P\times Q$ is a lattice polytope of dimension $p+q$.
Moreover, its Ehrhart polynomial is given by $L_{P\times Q}(t)= L_P(t)\cdot  L_Q(t)$.
By $\relint\bigl(m(P\times Q)\bigr)=\relint(mP)\times\relint(mQ)$, we have
$\codeg(P\times Q)=\max\{\codeg(P),\codeg(Q)\}$.
Hence, 
\[\deg h_{P\times Q}^* =p+q+1-\max\{p+1-\deg h_P^*,q+1-\deg h_Q^*\}.
\]

\begin{proposition}\label{prop:natural-products}
Let $P$ and $Q$ be lattice polytopes of dimensions $p$ and $q$.  For every
real $k$ and every $0\leq j\leq p+q$, we have 
\begin{equation}\label{eq:natural-product-convolution}
\A_j(P\times Q;k)=\sum_{\substack{a+b=j\\0\leq a\leq p\\0\leq b\leq q}}\A_a(P;k)\A_b(Q;k).
\end{equation}
Consequently,
\begin{small}
\begin{align*}
\rho(P\times Q)\leq\max\{\rho(P),\rho(Q)\},\ \ 
\tau(P\times Q)\leq\max\{\tau(P),\tau(Q)\},\ \ 
\tau^+(P\times Q)\leq\max\{\tau^+(P),\tau^+(Q)\}.
\end{align*}
\end{small}
\end{proposition}
\begin{proof}
For $L_{P\times Q}(t)$, writing $t=k+x$ gives
\[L_{P\times Q}(k+x)=\left(\sum_{a=0}^{p}\A_a(P;k)x^a\right)\left(\sum_{b=0}^{q}\A_b(Q;k)x^b\right).
\]
Taking the coefficient of $x^j$ proves \eqref{eq:natural-product-convolution}.
If both coefficient vectors are nonnegative at $k$, every summand in \eqref{eq:natural-product-convolution} is nonnegative.  If both vectors are positive, then for each $j$ there exists at least one pair $(a,b)$ with $a+b=j$, $0\leq a\leq p$, $0\leq b\leq q$,
and the corresponding product is positive.  By \cref{cor:positivity-persistence}, the three inequalities hold.
\end{proof}

\begin{proposition}\label{prop:natural-dilation}
Let $q\in\Z_{>0}$.  For every real $k$ and every $0\leq j\leq d$, we have 
\begin{equation}\label{eq:natural-dilation-coefficients}
 \A_j(qP;k)=q^j\A_j(P;qk).
\end{equation}
In particular,
\begin{equation}\label{eq:natural-dilation-rho}
 \rho(qP)=\frac{\rho(P)}q,
 \qquad
 \tau(qP)=\left\lceil\frac{\rho(P)}q\right\rceil.
\end{equation}
If $r=\rho(P)>0$, then $\tau^+(qP)=\lfloor\frac rq\rfloor+1$.
If $r=0$, then $\tau^+(qP)=0$ when $P$ is Taylor positive at $0$, and $\tau^+(qP)=1$ otherwise.
\end{proposition}
\begin{proof}
For every nonnegative integer $t$, we have $L_{qP}(t)=|(tq)P\cap\Z^d|=L_P(qt)$.
Setting $t=k+x$ yields
\[L_{qP}(k+x)=L_P(qk+qx)=\sum_{j=0}^{d}\A_j(P;qk)(qx)^j,
\]
which is \eqref{eq:natural-dilation-coefficients}.  Thus $qP$ is Taylor nonnegative at $k$ exactly when $P$ is Taylor nonnegative at $qk$. This proves \eqref{eq:natural-dilation-rho}; the formulas for $\tau$ and $\tau^+$ follow from \cref{prop:natural-rounding}.
\end{proof}

For a $d$-dimensional lattice polytope containing the origin in its interior, Hibi's palindromic theorem \cite{Hibi1992} (also see \cite[Theorem 4.6]{BeckRobins2015}) states that $P$ is reflexive if and only if $h_r^*=h_{d-r}^*$ for $0\leq r\leq d$.

\begin{theorem}\label{thm:palindromic-reflection}
Suppose $s=\deg h_P^*(z)$, $h_r^*=h_{s-r}^*$ ($0\leq r\leq s$), and set $\ell=d+1-s=\codeg(P)$.
For every $0\leq j\leq d$, we have 
\begin{equation}\label{eq:A-reflection}
\A_j(P;-\ell-k)=(-1)^{d-j}\A_j(P;k).
\end{equation}
Consequently:
\begin{enumerate}
 \item the zero multiset of $\A_j(P;\,\cdot\,)$ is invariant under $k\mapsto-\ell-k$;
 \item if $d-j$ is odd, then $\A_j\!\left(P;-\frac{\ell}{2}\right)=0$;
 \item if $\A_j(P;\,\cdot\,)$ is real-rooted and has positive degree, its smallest and largest zeros $\mu_j$ and $\lambda_j$ satisfy $\mu_j+\lambda_j=-\ell$.
\end{enumerate}
\end{theorem}
\begin{proof}
By \eqref{eq:hstar-expansion}, a direct calculation yields
\begin{equation}\label{eq:gorenstein-functional-equation}
 L_P(-\ell-t)=(-1)^dL_P(t).
\end{equation}
We obtain $(-1)^jL_P^{(j)}(-\ell-t)=(-1)^dL_P^{(j)}(t)$, 
or $L_P^{(j)}(-\ell-t)=(-1)^{d-j}L_P^{(j)}(t)$. Division by $j!$ proves \eqref{eq:A-reflection}.
If $\A_j(P;k)=0$, equation \eqref{eq:A-reflection} implies $\A_j(P;-\ell-k)=0$, with the same multiplicity.  Setting
$k=-\ell/2$ shows that $\A_j(P;-\ell/2)=0$ whenever $d-j$ is odd.
Finally, for a real-rooted coefficient polynomial the reflection sends its smallest zero to its largest zero, so $\lambda_j=-\ell-\mu_j$.
\end{proof}

A theorem of Rodriguez--Villegas \cite{RodriguezVillegas2002} gives a
particularly strong bridge from the roots of $h_P^*(z)$ to the roots of
$L_P(t)$.  Its application to Ehrhart polynomials is discussed in detail
by Braun and Liu \cite{BraunLiu2021}.

\begin{lemma}[\cite{RodriguezVillegas2002}]\label{thm:RV-ehrhart}
Suppose every zero of $h_P^*(z)$ lies on the unit circle.  Set $s=\deg h_P^*(z)$ and $\ell=d+1-s$.
Then there is a degree-$s$ polynomial $Q_P(t)$ such that
\begin{equation}\label{eq:RV-factorization}
L_P(t)=\left(\prod_{m=1}^{\ell-1}(t+m)\right)Q_P(t),
\end{equation}
and every zero of $Q_P$ has real part $-\ell/2$.
\end{lemma}

\begin{corollary}\label{cor:unit-circle-zero-center}
If every zero of $h_P^*(z)$ lies on the unit circle, then $\rho(P)=\tau(P)=\tau^+(P)=0$.
In fact, $\A_j(P;k)>0$ ($0\leq j\leq d,\ k\geq0$).
\end{corollary}
\begin{proof}
The roots supplied by the product in \eqref{eq:RV-factorization} are $-1,-2,\ldots,-(\ell-1)$,
and every remaining root has real part $-\ell/2$.  Since $\ell\geq1$, all roots of $L_P(t)$ lie in the open left half-plane.  The result now follows from \cref{cor:left-half-plane}.
\end{proof}

\section{Taylor positivity in natural classes}
\label{sec:natural-classes}

We derive bounds and exact values for several natural classes, and we analyze explicit non-Ehrhart-positive examples among order polytop, matroid polytope, generalized-permutohedra, and smooth polytope.
Whenever a lattice polytope is not full-dimensional, its Ehrhart polynomial and relative interior are understood with respect to the affine lattice in its affine span. This convention is particularly important for matroid base polytopes, which lie in affine hyperplanes.

\subsection{Order polytopes}

Let $\Pi$ be a finite poset with $n$ elements.  Following Stanley
\cite{Stanley1986}, its order polytope is
\[\mathcal O(\Pi)=\left\{x\in[0,1]^\Pi:x_u\leq x_v\text{ whenever }u\leq_\Pi v\right\}.
\]
It is an $n$-dimensional lattice polytope.  If $h(\Pi)$ denotes the largest number of elements in a chain of $\Pi$, then its $h^*$-degree has a particularly simple form.

\begin{proposition}\label{prop:natural-order-codegree}
For every finite poset $\Pi$ with $n$ elements, we have $\codeg\mathcal O(\Pi)=h(\Pi)+1$ and $\deg h_{\mathcal O(\Pi)}^*=n-h(\Pi)$.
Consequently,
\[\rho(\mathcal O(\Pi))\leq\tau^+(\mathcal O(\Pi))\leq\min\left\{\max\{0,n-h(\Pi)-1\},\left\lfloor\frac {n-1}{2}\right\rfloor\right\}.
\]
\end{proposition}
\begin{proof}
An integer point in $\relint(m\mathcal O(\Pi))$ is an integer-valued
function $f:\Pi\to\Z$ such that
\[0<f(u)<m\quad\text{for every }u\in\Pi,\qquad f(u)<f(v)\quad\text{whenever }u<_\Pi v.
\]
If $u_1<_\Pi u_2<_\Pi\cdots<_\Pi u_h$ is a chain with $h=h(\Pi)$ elements, then
\[ 1\leq f(u_1)<f(u_2)<\cdots<f(u_h)\leq m-1.
\]
There are only $m-1$ possible integer values, so $m\geq h+1$.
Conversely, define
\[f(v)=\max\{r:\text{there exists a chain } u_1<_\Pi\cdots<_\Pi u_r=v\}.
\]
Then $1\leq f(v)\leq h$.  If $u<_\Pi v$, a longest chain ending at
$u$ can be extended by $v$, and therefore $f(v)\geq f(u)+1$.
Thus $f$ is an interior lattice point of $(h+1)\mathcal O(\Pi)$.
This proves the codegree assertion.  Since the dimension is $n$, we have 
$\deg h_{\mathcal O(\Pi)}^*=n+1-\codeg\mathcal O(\Pi)=n-h(\Pi)$.
The bound for $\tau^+$ follows from \cref{cor:improved-combined-bound}.
\end{proof}

Disjoint unions of posets correspond to Cartesian products $\mathcal O(\Pi_1\sqcup\Pi_2)=\mathcal O(\Pi_1)\times\mathcal O(\Pi_2)$.
Thus \cref{prop:natural-products} applies directly.

\begin{corollary}\label{cor:natural-chain-unions}
Let $C_{a_i}$ be a chain with $a_i$ elements. If $\Pi=C_{a_1}\sqcup\cdots\sqcup C_{a_m}$ is a disjoint union of chains, then $\rho(\mathcal O(\Pi))=\tau(\mathcal O(\Pi))=\tau^+(\mathcal O(\Pi))=0$.
\end{corollary}
\begin{proof}
The order polytope of a chain with $a$ elements is a unimodular
$a$-simplex and $L_{\mathcal O(C_a)}(t)=\binom{t+a}{a}=\frac1{a!}\prod_{\nu=1}^{a}(t+\nu)$.
At a nonnegative center $k$, we have 
\[ L_{\mathcal O(C_a)}(k+x)=\frac1{a!}\prod_{\nu=1}^{a}(x+k+\nu),
\]
whose coefficients are all positive. This result follows from \cref{prop:natural-products}.
\end{proof}

The sharp low-dimensional classification is now known: every order
polytope of dimension at most $13$ is Ehrhart positive
\cite[Theorem~3.1]{LiuXinZhang2026}, whereas Liu and Tsuchiya constructed
a non-Ehrhart-positive order polytope in every dimension at least $14$
\cite[Theorem~1.7]{LiuTsuchiya2019}.  Hence every order polytope of
dimension at most $13$ satisfies $\rho=\tau=\tau^+=0$.

Let $Q_n$ be the poset consisting of one minimal element covered by $n$ pairwise incomparable elements.  Liu and Tsuchiya
\cite{LiuTsuchiya2019} studied this family in detail.  We use the Bernoulli polynomials $B_m(x)$ (see \cite{Book-Bernoulli}), defined by
\[\frac{ze^{xz}}{e^z-1}=\sum_{m\geq0}B_m(x)\frac{z^m}{m!},
\]
and write $B_m=B_m(0)$.  They satisfy
\begin{equation*}
B_m'(x)=mB_{m-1}(x),\qquad B_m(x+1)-B_m(x)=mx^{m-1}.
\end{equation*}

\begin{proposition}\label{prop:natural-star-poset}
For every $n\geq1$, we have 
\[L_{\mathcal O(Q_n)}(t)=\sum_{r=1}^{t+1}r^n=\frac{B_{n+1}(t+2)-B_{n+1}}{n+1}.
\]
For $1\leq j\leq n+1$, we have 
\begin{equation}\label{eq:natural-star-coefficients}
\A_j(\mathcal O(Q_n);k)=\frac{n!}{j!(n+1-j)!}\,B_{n+1-j}(k+2),
\end{equation}
whereas
\[\A_0(\mathcal O(Q_n);k)=\frac{B_{n+1}(k+2)-B_{n+1}}{n+1}.
\]
\end{proposition}
\begin{proof}
Fix a nonnegative integer $t$.  If the coordinate of the unique minimal
element is $a\in\{0,\ldots,t\}$, each of the $n$ maximal coordinates may
be chosen independently from $\{a,a+1,\ldots,t\}$.  Hence
\[L_{\mathcal O(Q_n)}(t)=\sum_{a=0}^{t}(t-a+1)^n=\sum_{r=1}^{t+1}r^n.
\]
By \cite[Remark 4.10]{Book-Bernoulli}, we have $\sum_{r=1}^{t+1}r^n=\frac{B_{n+1}(t+2)-B_{n+1}}{n+1}$.
For $j\geq1$, differentiate $j$ times and use
\[\frac{d^j}{dt^j}B_{n+1}(t+2)=\frac{(n+1)!}{(n+1-j)!}B_{n+1-j}(t+2).
\]
By \eqref{eq:formula-derivative}, we obtain \eqref{eq:natural-star-coefficients}.
\end{proof}

\subsection{Matroid base polytopes and generalized permutohedra}

Let $M$ be a matroid on $E=[n]$, with set of bases $\mathcal B(M)$.  Its
base polytope is
\[
 P(M)=\conv\{\mathbf 1_B:B\in\mathcal B(M)\}\subset\R^n.
\]
If $M=M_1\oplus M_2$ is a direct sum on disjoint ground sets, then
\begin{equation*}
 P(M_1\oplus M_2)=P(M_1)\times P(M_2).
\end{equation*}
Thus all conclusions of \cref{prop:natural-products} apply to matroid
direct sums.

Uniform matroid base polytopes, equivalently hypersimplices, are Ehrhart positive by Ferroni \cite{Ferroni2021Hypersimplices}.  All rank-two matroid base polytopes are Ehrhart positive by Ferroni, Jochemko, and Schr\"oter \cite{FerroniJochemkoSchroeter2022}.  Hence these polytopes, and direct sums of such polytopes, satisfy
\[\rho(P(M))=\tau(P(M))=\tau^+(P(M))=0.
\]

A generalized permutohedron may equivalently be defined as a polytope all of whose edge directions are parallel to vectors $e_i-e_j$. Matroid base polytopes are integral generalized permutohedra.  Postnikov's lattice-point formula \cite[Theorem~11.3]{Postnikov2009} proves Ehrhart positivity for $\mathcal Y$-generalized permutohedra, namely integral
Minkowski sums
\[\sum_{\emptyset \neq I\subseteq[n]}y_I\Delta_I\qquad\text{with }y_I\in\Z_{\geq0},
\]
where $\Delta_I=\conv\{e_i:i\in I\}$ is a standard coordinate simplex.
They therefore satisfy $\rho=\tau=\tau^+=0$.  The same is true for every integral generalized
permutohedron of dimension at most six by the Berline--Vergne positivity result of Castillo and Liu \cite{CastilloLiu2018}.

The general class is not Ehrhart positive.  Ferroni constructed connected matroids with negative Ehrhart coefficients \cite{Ferroni2022Matroids}. We show that the first explicit counterexample in that paper has least
integral centers equal to one.

For integers $1\leq k\leq n-1$, set
\begin{align*}
 U_{k,n}(t)&=\sum_{a=0}^{k-1}(-1)^a\binom{n}{a}\binom{(k-a)t+n-1-a}{n-1},\\
 D_{k,n}(t)&=\frac{\binom{t+n-k}{n-k}}{\binom{n-1}{k-1}}\sum_{a=0}^{k-1}\binom{n-k-1+a}{a}\binom{t+a}{a}.
\end{align*}
Here every binomial coefficient is interpreted as a polynomial in $t$.
The first formula, originally due to Katzman \cite[Corollary~2.2]{Katzman2005}, is the Ehrhart polynomial of the
hypersimplex $P(U_{k,n})$. The second is the Ehrhart polynomial of the minimal connected matroid $T_{k,n}$; see
\cite{Ferroni2022Minimal,Ferroni2022Matroids}.

\begin{example}\label{thm:natural-matroid-center}
Ferroni \cite[Theorem~5.3]{Ferroni2022Matroids} constructs a sparse paving matroid $M$ of rank $9$ on $20$ elements with $8398$
circuit-hyperplanes.  It is connected, and its Ehrhart polynomial is
\begin{equation*}
 L_{P(M)}(t)=U_{9,20}(t)-8398D_{9,20}(t-1).
\end{equation*}
Ferroni proved that $P(M)$ is non-Ehrhart positive; in fact, its quadratic and cubic Ehrhart coefficients are negative, so it is not Taylor nonnegative at zero.
A tedious computation shows that $\rho(P(M))\in(0,1)$, $\tau(P(M))=\tau^+(P(M))=1$.
\end{example}

\begin{remark}
There is also a noteworthy positivity statement at a negative center.
If $M$ has rank $k$, has $n$ elements, and has no loops or coloops, then
Chavez, Dorpalen-Barry, Ferroni, Liu, Rinc\'on, and Vindas-Mel\'endez
proved \cite[Theorem~1.1]{ChavezEtAl2026} that
\[\A_1(P(M);-1)=L_{P(M)}'(-1)=\frac{\beta(M)}{(n-1)\binom{n-2}{k-1}},
\]
where $\beta(M)$ is Crapo's beta invariant \cite{Crapo1967}. This result does not by
itself determine the least nonnegative centers, but it shows that the
shifted coefficient functions of matroid polytopes already encode
classical matroid invariants outside the interval used in
\cref{def:positivity-indices}.
\end{remark}

\subsection{Smooth polytopes}

A $d$-dimensional lattice polytope is smooth if, at every vertex, the
primitive directions of the $d$ incident edges form a lattice basis.
Smoothness is preserved by integral dilations.  It is also preserved by
Cartesian products: at a vertex $(v,w)$ of $P\times Q$, the primitive
edge directions are
\[
 (u_1,0),\ldots,(u_p,0),(0,z_1),\ldots,(0,z_q),
\]
where the $u_i$ and $z_j$ are lattice bases at $v$ and $w$; their union
is a lattice basis of the product lattice.

Castillo, Liu, Nill, and Paffenholz constructed smooth polytopes with all potentially negative Ehrhart coefficients simultaneously negative \cite{CastilloEtAl2018}.  Their three-dimensional building blocks are
\[
 B_k=\operatorname{chisel}\!
 \left(3^k[0,1]^3;(3^{k-1},3^{k-2},\ldots,3,1)\right).
\]
For a detailed definition of $\operatorname{chisel}$, see \cite{CastilloEtAl2018}.
The Ehrhart polynomial of $B_k$ is 
\begin{equation*}
 L_{B_k}(t)=q_3t^3+q_2t^2-q_1t+1,
\end{equation*}
where
\begin{align*}
 q_1=3^{k-2}(8k-27),\quad
 q_2=3^{k-1}(7\cdot3^k+2),\quad
 q_3=\frac12\,3^{k-2}(17\cdot3^{2k}+1).
\end{align*}
For $k\geq4$, all three $q_i$ are positive.

\begin{theorem}\label{thm:natural-smooth-center-one}
In every dimension $d\geq3$, there exists a smooth $d$-dimensional lattice polytope $P_d$ such that every Ehrhart coefficient not forced to be positive is negative, but $\tau(P_d)=\tau^+(P_d)=1$.
\end{theorem}
\begin{proof}
First consider $B_k$ with $k\geq4$. We have 
\[
 L_{B_k}(1+x)
 =L_{B_k}(1)+(3q_3+2q_2-q_1)x
 +(3q_3+q_2)x^2+q_3x^3.
\]
The constant term is the positive integer $L_{B_k}(1)$.  The quadratic
and cubic coefficients are plainly positive.  For the linear coefficient,
\begin{align*}
3q_3+2q_2-q_1=3^{k-2}\left(\frac{51}{2}3^{2k}+42\cdot3^k+\frac{81}{2}-8k\right)>0.
\end{align*}
Thus $B_k$ is Taylor positive at $1$.
For $n\geq1$, Castillo, Liu, Nill, and Paffenholz take products $P^{n}(k,a)=B_k\times a[0,1]^n$.
The second factor has Ehrhart polynomial $(at+1)^n$, and at $t=1+x$ it becomes $(ax+a+1)^n$, whose coefficients are all positive.
Hence \cref{prop:natural-products} shows that $P^{n}(k,a)$ is Taylor positive at $1$.

For each $n$, the parameters $k$ and $a$ in \cite[Theorem~1.2]{CastilloEtAl2018} are chosen so that every coefficient
of degrees $1,\ldots,n+1$ in the ordinary Ehrhart polynomial is negative. Therefore these polytopes are not Taylor nonnegative at zero.  Their dimension is $n+3$, and $\tau(P^n(k,a))=\tau^+(P^n(k,a))=1$.  In dimension three one may take $B_4$, whose linear Ehrhart coefficient is negative.
\end{proof}

\begin{example}\label{thm:natural-smooth-real-center}
The smooth three-dimensional polytope $B_4$ has
\[L_{B_4}(t)=501921t^3+15363t^2-45t+1.
\]
Through a tedious computation, we obtain
\begin{equation*}
 \rho(B_4)=\frac{72\sqrt{234399}-30726}{3011526}=0.001372272642144697\ldots.
\end{equation*}
In particular, $\tau(B_4)=\tau^+(B_4)=1$.
By \cref{prop:natural-rounding,prop:natural-dilation}, for every $q\in\Z_{>0}$, we have 
\[\rho(qB_4)=\frac{\rho(B_4)}q,\qquad\tau(qB_4)=\tau^+(qB_4)=1.
\]
Integral dilation does not change the primitive edge directions, so every $qB_4$ remains smooth.
\end{example}

\begin{corollary}\label{cor:natural-no-real-gap}
Within both the class of smooth lattice polytopes and the class of integral generalized permutohedra, there are non-Ehrhart-positive polytopes for which examples with $\rho(P)>0$ and $\rho(P)\rightarrow 0$, while $\tau(P)=\tau^+(P)=1$.
\end{corollary}
\begin{proof}
For smooth polytopes, take the sequence $qB_4$ ($q\geq 1$) and use \cref{thm:natural-smooth-real-center}.  For generalized
permutohedra, take integral dilations of the matroid polytope in \cref{thm:natural-matroid-center}.  Dilations of a generalized
permutohedron are generalized permutohedra.  By \cref{prop:natural-dilation}, their least real centers are $\rho(P(M))/q$,
whereas their ordinary Ehrhart coefficient signs are unchanged and their least integral centers remain one.
\end{proof}

The examples exhibit several different behaviors.  Disjoint unions of chains, hypersimplices, rank-two matroid polytopes, and
$\mathcal Y$-generalized permutohedra satisfy $\rho=\tau=\tau^+=0$.  On the other hand, order polytopes, matroid base polytopes, generalized permutohedra, and smooth polytopes all contain examples with $\tau=\tau^+=1$. 
According to \cref{cor:natural-no-real-gap}, the invariant $\rho$ can be arbitrarily close to zero while both integer indices remian equal to one.

\section{Concluding remarks}\label{sec:conclusion}

We introduced the Taylor coefficients $A_j(P; k)$ of lattice polytopes, derived explicit formulas for them, determined sharp universal bounds for their least nonnegative centers, and established several structural properties. In particular, these results provide a systematic partial answer to \cref{Problem-Join-open}. 

Several classification problems remain open.  It would be useful to characterize natural classes for which $\rho(P)=0$, especially within the families of matroid base polytopes and smooth polytopes.  One may also ask for a structural description of the polytopes attaining $\tau(P)=\lfloor(d-1)/2\rfloor$, beyond the generalized Reeve constructions and their pyramids.






\noindent
{\small \textbf{Acknowledgments:}}
The authors would like to express sincere gratitude for all the suggestions that have improved the presentation of this paper.
Feihu Liu was partially supported by the Postdoctoral Fellowship Program and China Postdoctoral Science Foundation (Grant No. BX2026002).

%

\end{document}